\documentclass[a4,11pt]{amsart}
\usepackage[top=3cm, bottom=3cm, left=2.5cm, right=2.5cm]{geometry}

\usepackage{amssymb,amsmath,amsthm,txfonts, amsfonts}
\usepackage{amscd}
\usepackage[all]{xy} 
\usepackage[colorlinks=true,citecolor=blue,linkcolor=blue]{hyperref}

\usepackage{graphicx}
\usepackage{amsmath,amsfonts}
\usepackage{bm}
\usepackage{physics}
\usepackage{enumerate}
\usepackage{color,soul}
\usepackage{hyperref}

\usepackage{mathrsfs}
\usepackage{color}
\usepackage{cite}
\usepackage{nicefrac}
\usepackage{hyperref}
\usepackage{textcomp}
\usepackage{protosem}

\usepackage{graphicx}
\usepackage{lscape}
\usepackage[hang,small,bf]{caption}
\usepackage[subrefformat=parens]{subcaption}
\usepackage{natbib}

\newtheorem{thm}{Theorem}[section]
\newtheorem{lem}[thm]{Lemma}

\newtheorem{prop}[thm]{Proposition}
\theoremstyle{definition}
\newtheorem{defn}[thm]{Definition}
\theoremstyle{remark}
\newtheorem{rem}[thm]{\textbf{Remark}}

\newtheorem{q}[thm]{\textbf{Question}}

 \makeatletter
    
    \@addtoreset{equation}{section}
  \makeatother

\makeatletter
      \def\@makefnmark{%
         \leavevmode
            \raise.9ex\hbox{\check@mathfonts
                \fontsize\sf@size\z@\normalfont%
                            \@thefnmark}%
       }
      \makeatother

\makeatletter
\@namedef{subjclassname@2020}{\textup{2020} Mathematics Subject Classification}
\makeatother

\begin{document}

\title[Stability of Lamb dipoles without the finite mass condition]{Stability of Lamb dipoles for odd-symmetric and non-negative initial disturbances without the finite mass condition}
\author{Ken Abe, Kyudong Choi, In-Jee Jeong}
\date{}
\address[Ken Abe]{Department of Mathematics, Graduate School of Science, Osaka Metropolitan University, 3-3-138 Sugimoto, Sumiyoshi-ku Osaka, 558-8585, Japan}
\email{kabe@omu.ac.jp}
\address[Kyudong Choi]{Department of Mathematical Sciences, Ulsan National Institute of Science and Technology, UNIST-gil 50, Ulsan, 44919, Republic of Korea}
\email{kchoi@unist.ac.kr}
\address[In-Jee Jeong]{Department of Mathematical Sciences and RIM, Seoul National University, Seoul 08826, Korea}
\email{injee\_j@snu.ac.kr}

\date{\today}

\begin{abstract}
In this paper, we consider the stability of the Lamb dipole solution of the two-dimensional Euler equations in $\mathbb{R}^{2}$ and question under which initial disturbance the Lamb dipole is stable, motivated by experimental work on the formation of a large vortex dipole in two-dimensional turbulence. We assume (O) odd symmetry for the $x_2$-variable and (N) non-negativity in the upper half plane for the initial disturbance of vorticity, and establish the stability theorem of the Lamb dipole without assuming (F) finite mass condition. The proof is based on a new variational characterization of the Lamb dipole using an improved energy inequality.
\end{abstract}

\maketitle

%\tableofcontents

%\kyu{I added corollary 1.5. Does it make sense?}

\section{Introduction}

\subsection{Lamb dipoles}
We consider the two-dimensional Euler equations in $\mathbb{R}^{2}$ expressed in the vorticity form 
\begin{equation}
\begin{aligned}
\partial_t \zeta+v\cdot \nabla \zeta=0,\quad  v&=k*\zeta,
\end{aligned}
\label{eq: Euler}
\end{equation}
with the kernel $k(x)=(2\pi)^{-1}x^{\perp}|x|^{-2}$ for $x^{\perp}={}^{t}(-x_2,x_1)$. The equations \eqref{eq: Euler} admit traveling wave solutions of the form 
\begin{equation}
\begin{aligned}
v(x,t)&=u(x+u_{\infty}t)-u_{\infty},\\
\zeta(x,t)&=\omega(x+u_{\infty}t),    
\end{aligned}
\label{eq: TW}
\end{equation}
for a constant $u_{\infty}\in \mathbb{R}^{2}$ with the profile $(u,\omega)$ satisfying the stationary equations 
\begin{equation}
\begin{aligned}
u\cdot \nabla \omega=0,\quad  u=k*\omega+u_{\infty}.\\    
\end{aligned}
\label{eq: SE}
\end{equation}
The simplest solution to \eqref{eq: SE} is a Lamb dipole (Chaplygin--Lamb dipole) \cite{Lamb2nd}, \cite{Chap1903}, \cite{Lamb3rd}, \cite[p.231]{Lamb} which is symmetric about the $x_1$-axis; see \cite[p.197]{MV94} for its origin.

\begin{defn}[Lamb dipole]
Let $0<\lambda, W<\infty$. We say that $\omega_{L}=\omega_{L}^{\lambda,W}$ is a Lamb dipole if $\omega_{L}=\lambda \max\{\Psi_{L},0 \}$ for 
\begin{equation}
\Psi_{L}(x)=\left\{
\begin{aligned}
C_LJ_1(\sqrt{\lambda} r)\sin\theta,\quad r\leq a,\\
-W\left(r-\frac{a^{2}}{r}\right)\sin\theta,\quad r>a,
\end{aligned}
\right. \label{eq: Lamb}
\end{equation}
in the coordinates $(r,\theta)$ with the constants
\begin{align}
C_L=-\frac{2W}{\sqrt{\lambda}J_0(c_0)},\quad a=\frac{c_0}{\sqrt{\lambda}},   \label{eq: LambConst}
\end{align}
where $J_{m}(r)$ is the $m$-th order Bessel function of the first kind and $c_0=3.8317\cdots$ is the first zero point of $J_1$, i.e., $J_1(c_0)=0$.
\end{defn}

\begin{figure}[h]
  \centering
  \begin{subfigure}[t]{0.45\textwidth}\centering
    \raisebox{2.5ex}{
    \resizebox{0.55\linewidth}{!}
{ \includegraphics[width=0.7\linewidth]{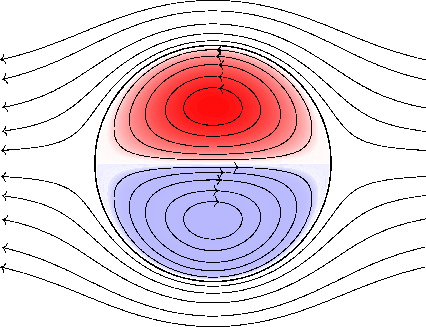}}}
    \caption{Chaplygin--Lamb dipole}
  \end{subfigure}\hspace{-30pt}
  \begin{subfigure}[t]{0.45\textwidth}\centering
    \resizebox{0.55\linewidth}{!}{ \includegraphics[width=0.7\linewidth]{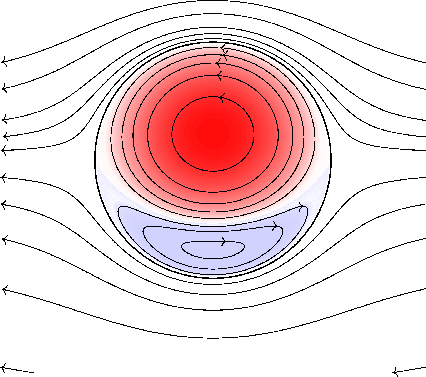}}
    \caption{Chaplygin's asymmetric dipole}
  \end{subfigure}
  \caption{Streamlines of symmetric and asymmetric dipoles. Positive vorticity in red and negative vorticity in blue.}\label{f: CL}
\end{figure}

The Lamb dipole \eqref{eq: Lamb} satisfies the equations \eqref{eq: SE} with the associated velocity field $u_{L}={}^{t}(\partial_{x_2}\Psi_{L},-\partial_{x_1}\Psi_{L})=(\partial_{x_2}\psi_{L}-W,-\partial_{x_1}\psi_{L})$ and the constant $u_{\infty}={}^{t}(-W,0)$. Its kinetic energy, enstrophy, and impulse are as follows:
\begin{align}
E=\frac{1}{2}\int_{\mathbb{R}^{2}_{+}}|\nabla \psi_{L}|^{2}dx=\frac{c_0^{2}\pi W^{2}}{\lambda},\quad
Z=\int_{\mathbb{R}^{2}_{+}}| \omega_{L}|^{2}dx=c_0^{2}\pi W^{2},\quad
P=\int_{\mathbb{R}^{2}_{+}}x_2 \omega_{L}dx=\frac{c_0^{2}\pi W}{\lambda}.  \label{eq: C}
\end{align}
We remark that Chaplygin \cite{Chap1903} derived asymmetric dipoles, including the Chaplygin--Lamb dipole as a particular case; see Figure \ref{f: CL}.

%The quantities $Z$ and $P$ can be computed from $\omega_L=\lambda (\Psi_{L})_{+}$ for $t_+=\max\{t,0\}$. The kinetic energy can be computed from $2E=Z/\lambda+WP$ by using $||\nabla \psi_L||_{L^{2}}^{2}=||\psi_L\omega_L||_{L^{1}}$.

The solution \eqref{eq: Lamb} is a theoretical model for coherent structures in two-dimensional turbulence, e.g., \cite{CH09}. It possesses the following properties:
\begin{itemize}
\item[(O)] \textbf{Odd-symmetry}; $\zeta(x_1,x_2)=-\zeta(x_1,-x_2)$   
\item[(N)] \textbf{Non-negativity}; $\zeta(x_1,x_2)\geq 0$ for $x_2\geq 0$ 
\item[(F)] \textbf{Finite mass}; $\zeta \in L^{1}(\mathbb{R}^{2})$   
\end{itemize}
It is observed from experimental works \cite{VF89}, \cite{FV94}, \cite{Afan} that large dipole vortices are formed as stable structures in stratified flows for quite general initial data; see Figure \ref{f: VD}. On the other hand, the mathematical stability theorems of the Lamb dipole \eqref{eq: Lamb} in the 2D Euler equations \eqref{eq: Euler} \cite[Theorem 1.1]{AC22}, \cite[Theorem 5.1]{Wang24} require the restrictive conditions (O), (N), and (F) for the initial disturbance $\zeta_0$. It is a question of which initial disturbances make the solution \eqref{eq: Lamb} stable. We address this question in the following:
\begin{q}\label{q: Q}
For which initial disturbances is the Lamb dipole \eqref{eq: Lamb} stable in the 2D Euler equations \eqref{eq: Euler}? 
\end{q}

\subsection{The statement of the main result}

In this paper, we note that the Lamb dipole \eqref{eq: Lamb} is stable in the 2D Euler equations \eqref{eq: Euler} without assuming the finite mass condition (F) for initial disturbances $\zeta_0$. We assume the boundedness of the disturbance $\zeta_0\in L^{2}(\mathbb{R}^{2})$ and $x_2|\zeta_0|\in L^{1}(\mathbb{R}^{2})$ and consider the stability of \eqref{eq: Lamb} for solutions to the Euler equations \eqref{eq: Euler} with finite kinetic energy, enstrophy, and impulse. The following main result improves the stability result of \cite{AC22}.

%thm1.1
\begin{thm}\label{t: mthm}
Let $0<\lambda, W<\infty$ and $P=c_0^{2}\pi W/\lambda$. The Lamb dipole $\omega_{L}$ is orbitally stable in the sense that for $\varepsilon>0$, there exists $\delta>0$ such that for $\zeta_0\in L^{2} (\mathbb{R}^{2}_{+})$ satisfying $x_2\zeta_{0}\in L^{1}(\mathbb{R}^{2}_{+})$, $\zeta_{0}\geq 0$, 
\begin{align}
\inf_{y\in \partial\mathbb{R}^{2}_{+}}\left\|\zeta_0-\omega_{L}(\cdot+y) \right\|_{L^{2}(\mathbb{R}^{2}_{+})}+\left|\int_{\mathbb{R}^{2}_{+}} x_2\zeta_0\dd x-P  \right|  \leq \delta, \label{eq: S1}
\end{align}
there exists a global weak solution $\zeta(t)$ of \eqref{eq: Euler} satisfying
\begin{align}
\inf_{y\in \partial\mathbb{R}^{2}_{+}}\left\{\left\|\zeta(t)-\omega_{L}(\cdot+y) \right\|_{L^{2}(\mathbb{R}^{2}_{+})}+\left\|x_2(\zeta(t)-\omega_{L}(\cdot+y)) \right\|_{L^{1}(\mathbb{R}^{2}_{+})}\right\}\leq \varepsilon,\quad \textrm{for all}\ t\geq0.  \label{eq: S2}
\end{align}
\end{thm}

\begin{figure}[h]
  \centering
\rotatebox{90}{  \includegraphics[scale=0.15]{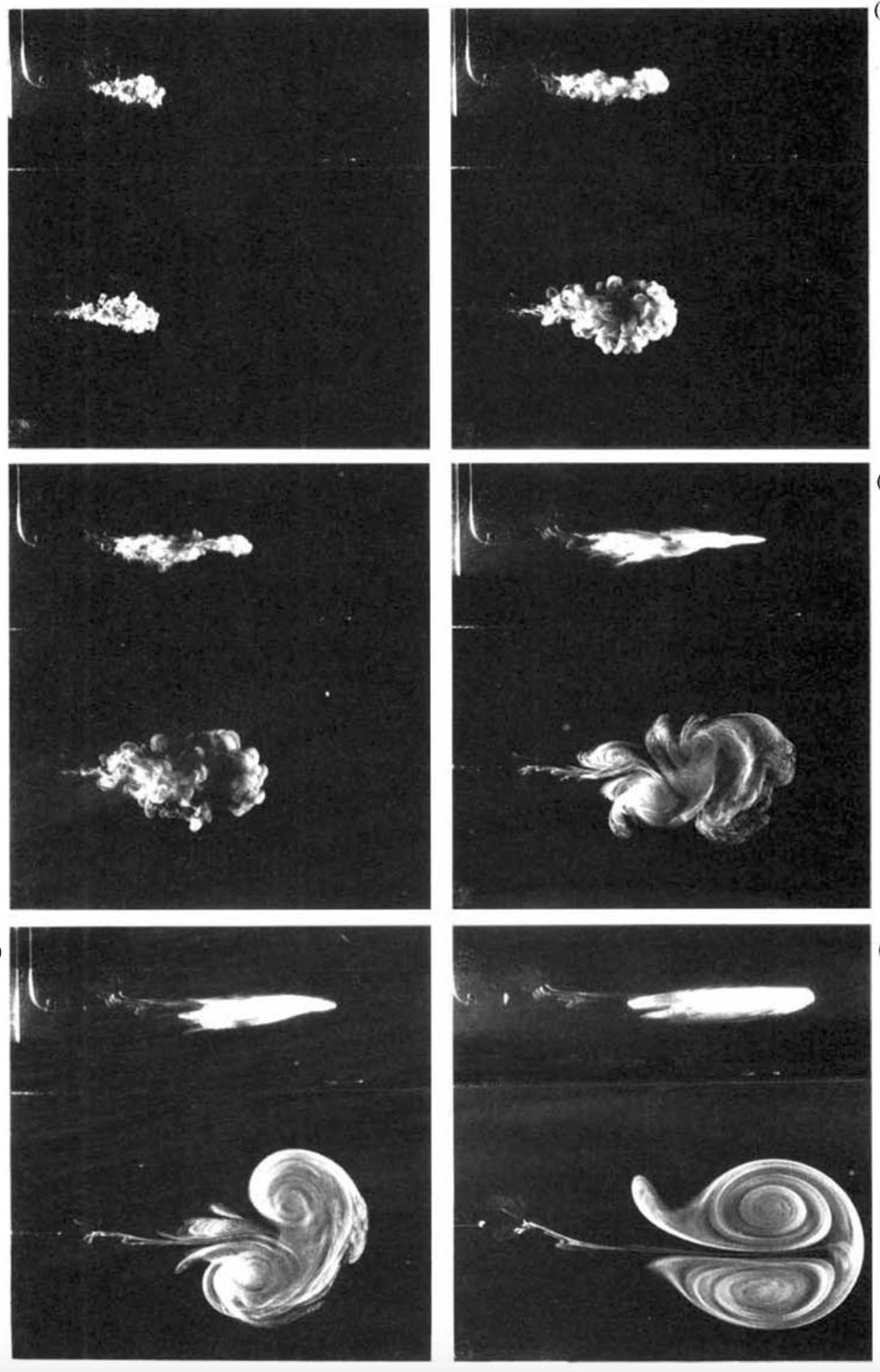}} % 同じフォルダに VD.jpg を置く
  \caption{The emergence of a dipole vortex in a stratified flow created by a pulsed horizontal injection. Each photograph presents a top view and a side view. From \cite{FV94}. The figure has been rotated by $90^\circ$. Licensed under CC BY 4.0. }
    \label{f: VD}
\end{figure}

The orbital stability of traveling wave solutions to the 2D Euler equations \eqref{eq: Euler} was first established in Burton--Lopes--Lopes \cite{BNL13} for a large class of vortex-pairs by using a rearrangement of functions. Burton \cite{B21} showed the orbital stability of vortex pairs by using a rearrangement with the norm $||\zeta||_{L^{p}\cap L^{1}(\mathbb{R}^{2}_{+})}+|\int_{\mathbb{R}^{2}_{+}} x_2 \zeta dx|$ for $p>2$ by assuming (O), (N), and the compactness of the support of $\zeta_0$. Wang \cite[Theorem 5.1]{Wang24} deduced the orbital stability of the Lamb dipole \eqref{eq: Lamb} from the stability result of \cite{B21} and the variational characterization of \cite{Burton05b}. More recently, the work  \cite[Theorem 1.2]{Wang25} showed the orbital stability of a truncated Lamb dipole in a unit disk (in the sense of up to rotation) for a general initial disturbance without assuming odd symmetry (O) and non-negative conditions (N). One of the difficulties in removing the conditions (O) and (N) is the lack of variational formulations for vortex pairs without using those conditions in $\mathbb{R}^{2}$.

The orbital stability of vortex pairs has been obtained as the stability of a set of minimizers (or maximizers) for a certain variational problem, and it is, in general, a question of whether a set of minimizers is a translation of a unique minimizer. The classical rigidity theorems establish the uniqueness of large vortex pairs and vortex rings, such as the Lamb dipole \cite{Burton96}, \cite{Burton05b}, Hill's spherical vortex \cite{AF86}, and Norbury's rings \cite{AF88}. The work \cite{Choi24} showed the stability of Hill's spherical vortex in the axisymmetric Euler equations without swirls, assuming the finite mass condition for initial disturbances; see also \cite{CQZZ2} for the stability of Norbury's rings. 

Recently, Cao--Qin--Zhang--Zhou \cite[Theorem 1.8]{CQZZ} established the uniqueness of concentrated vortex pairs and deduced their orbital stability from the stability result of \cite{BNL13}. See also Cao--Lai--Qin--Zhang--Zhou \cite[Theorem 1.2]{CLQZZ} for uniqueness and stability of thin-cored axisymmetric vortex rings without swirls. It is a question of whether the condition (F) can be removed for the stability of vortex pairs other than the Lamb dipole \eqref{eq: Lamb}, cf. \cite[Theorem 1.4]{AC22}.

We also mention the recent work \cite{DG} on long-time approximation of small viscous flows originating from point vortex pairs by viscous dipole solutions (two Lamb--Oseen vortices). See also \cite{Ga11}. 

\subsection{Research on dipole vortices}

Let us briefly discuss dipole vortices and the long-time behavior of solutions to the 2D Euler equations.

\subsubsection{Physical backgrounds}

In geophysical fluid dynamics, large dipole vortices are called modons \cite{Stern}. There exist modon solutions (including \eqref{eq: Lamb} as a particular case) also in the beta-plane equations and quasi-geostrophic shallow water/Charney--Hasegawa--Mima equations \cite{LR76}, \cite[5.6]{PP}. Modons also exist for the Euler equations on a rotating sphere; see \cite{PG15} for a review. The stability theory for the Euler equations on a rotating sphere has been developed for linear wave solutions (Rossby--Haurwitz waves) in \cite{Taylor16}, \cite{CG}, \cite{CWZ}, \cite{CGLZ}.

\subsubsection{Large vortex dynamics}

In general, describing long-time dynamics of solutions to the 2D Euler equations is a highly challenging problem. In the specific setting of the half-plane, there are a few general bounds on the large-scale features of solutions \cite{ISG99}, \cite{ILN03}. While there are large classes of traveling wave solutions, it is a highly non-trivial problem to demonstrate the existence of global-in-time solutions with non-trivial dynamical behavior. The existence of solutions converging to a separating pair of dipoles as $t\to\infty$ was obtained in \cite{DdPMP2} by the gluing method for the Euler equations; see also \cite{DdpMW}, \cite{DdPMP}. Moreover, the existence of time-periodic leapfrogging patches was proved in \cite{BHM}. 
On the other hand, the works \cite{CJY} and \cite{AJY} show the stability of multi-vortex solutions. Namely, there exist global-in-time unique solutions whose vorticity is concentrated on two separating Lamb dipoles \cite{CJY} and a chain of $N$ Lamb dipoles with no collisions \cite{AJY}.

\subsubsection{Numerical works}

There is quite a large literature on Lamb dipoles from computational and experimental fluid dynamics \cite{FV94}, \cite{Or92}, \cite{VF89}, \cite{NiRa}, \cite{KrXu21}, \cite{Protas}. The work \cite{NiRa} performed a numerical simulation of the Navier--Stokes equations for general initial data with nonzero impulse and observed the creation of a dipole structure, which is quite similar to the Lamb dipole \eqref{eq: Lamb}. While the Lamb dipole is not an exact traveling wave solution of the Navier--Stokes equations, the work \cite{NiRa} obtains a theoretical time-dependent ansatz of the viscous Lamb dipole by letting parameters $\lambda, W$ change in time, and shows that it is in remarkable agreement with results from direct numerical simulations. Recently, the work \cite{KrXu21} performed high-resolution numerical computations for the Lamb dipole in a large range of Reynolds numbers and studied the effects of convection on the dipole evolution. All of these numerical studies confirm filamentation behavior, creation of long and thin tails, emerging behind the Lamb dipole, cf. Figure \ref{f: VD}. 

\subsubsection{Small-scale formations}

The work \cite{CJ-Lamb} investigated the filamentation near the Lamb dipole \eqref{eq: Lamb}. It estimated the speed of the perturbations of \eqref{eq: Lamb} in the stability estimate (namely, $y$ in \eqref{eq: S2}) and proved linear-in-time filamentation for arbitrarily small and localized perturbations of \eqref{eq: Lamb}. In particular, the result in \cite{CJ-Lamb} gives infinite-time linear growth of the $W^{1,p}$-norm of the vorticity for all $1 \le p \le \infty$, showing instability of the Lamb dipole in $W^{1,p}$. More recent work \cite{JYZ} obtained superlinear growth of the $W^{1,\infty}$-norm for perturbations of \eqref{eq: Lamb} following the ideas of Denisov \cite{Den09} by using hyperbolic stagnation points in the moving frame. 
    
    In general, one may ask how fast the $W^{1,\infty}$-norm of the vorticity can grow in time for smooth initial data. Remarkably, Zlato\v{s} \cite{zlatos2025} recently obtained the optimal double exponential growth for the $W^{1,\infty}$-norm on the half-plane. Prior to this work, the double exponential growth rate was achieved only in bounded domains \cite{KS}, \cite{Xu}.

A relevant question to the filamentation is the regularity of solutions between two touching dipole vortices. The works \cite{Choi25} and \cite{HT} showed the existence of touching traveling dipole patches, in contrast to the touching continuous dipole \eqref{eq: Lamb}.

\subsubsection{Non-uniquness} We note that recently the Lamb dipole \eqref{eq: Lamb} was used in \cite{BCK} as a building block of a convex integration scheme for the 2D Euler equations in a periodic domain, resulting in the first non-uniqueness example with integrable vorticity. 

\subsection{The idea of the proof: the new energy inequality}

We show Theorem \ref{t: mthm} by a new variational characterization of  \eqref{eq: Lamb} without using mass. A heuristic idea is a dimensional balance between three quantities $E$, $Z$, and $P$ in \eqref{eq: C}\footnote{The dimensional balance $[E]=\sqrt{[Z]}[P]$ holds for all 2D flows since $[E]=L^{4}/T^{2}$, $[Z]=L^{2}/T^{2}$, $[P]=L^{3}/T$ for the length $L$ and time $T$ by $[u]=L/T$, $[\omega]=1/T$, and $[dx]=L^{2}$.}. Namely, by $Z/(2\lambda)=WP/2$ and $P=c_0^{2}\pi W/\lambda$, 
\begin{align*}
E=\frac{1}{2\lambda}Z+\frac{W}{2}P=\sqrt{\frac{Z}{\lambda}}\sqrt{WP}=\frac{1}{c_0\sqrt{\pi}}\sqrt{Z}P.
\end{align*}
By using the norms,   
\begin{align}
||\nabla \psi_{L}||_{L^{2}(\mathbb{R}^{2}_{+})}
=\sqrt{\frac{2}{c_0\sqrt{\pi}}}||\omega_{L}||_{L^{2}(\mathbb{R}^{2}_{+})}^{\frac{1}{2}}||x_2\omega_{L}||_{L^{1}(\mathbb{R}^{2}_{+})}^{\frac{1}{2}}.   \label{eq: ID}
\end{align}
In the recent work \cite[Corollary 2.5]{AJY}, the following new energy inequality was obtained
\begin{align}
||\nabla \psi ||_{L^{2}(\mathbb{R}^{2}_{+})}\leq C ||\omega||_{L^{2}(\mathbb{R}^{2}_{+})}^{\frac{1}{2}}||x_2\omega||_{L^{1}(\mathbb{R}^{2}_{+})}^{\frac{1}{2}},     \label{eq: IEI}
\end{align}
for $\omega \in L^{2}(\mathbb{R}^{2}_{+})$ such that $x_2\omega\in L^{1}(\mathbb{R}^{2}_{+})$ and $\psi=(-\Delta_{D})^{-1}\omega$ with some constant $C$ by using the Green function. The inequality \eqref{eq: IEI} holds for all $\omega$ with the constant
\begin{align}
C_{*}=\sqrt[4]{\frac{3}{8\pi}},   \label{eq: HLS0}
\end{align}
by the Hardy-Littlewood-Sobolev inequality in $\mathbb{R}^{4}$ and the isometry between homogeneous Sobolev spaces on $\mathbb{R}^{2}_{+}$ and $\mathbb{R}^{4}$. The energy inequality \eqref{eq: IEI} enables one to formulate the following  minimization problem without using a mass constraint, cf. \cite{AC22}: 
\begin{align}
\mathcal{I}_{\mu,\lambda}=\inf_{K_{\mu}}I_{\lambda},  \label{eq: MinP}
\end{align}
for the functional  
\begin{align*}
I_{\lambda}[\omega]=\frac{1}{2\lambda}\int_{\mathbb{R}^{2}_{+}}\omega^{2}dx-\frac{1}{2}\int_{\mathbb{R}^{2}_{+}}|\nabla \psi|^{2}dx,\quad \psi=(-\Delta_D)^{-1}\omega,\quad \lambda>0,
\end{align*}
and the admissible set   
\begin{align*}
K_{\mu}=\left\{\omega\in L^{2}(\mathbb{R}^{2}_{+}) \ \middle|\ \int_{\mathbb{R}^{2}_{+}}x_2\omega dx=\mu,\ \omega\geq 0\  \right\},\quad \mu>0.
\end{align*}
Namely, $\mathcal{I}_{\mu,\lambda}$ is bounded from below for $\lambda,\mu>0$ thanks to \eqref{eq: IEI}. Our main task is to show that all minimizers to \eqref{eq: MinP} are translations of the Lamb dipoles \eqref{eq: Lamb} for the $x_1$-variable, and the minimum is the constant 
\begin{align}
\mathcal{I}_{\mu,\lambda}=-\frac{1}{2c_0^{2}\pi}\mu^{2}\lambda.  \label{eq: MinM}
\end{align}
The minimum \eqref{eq: MinM} provides a sharp constant of \eqref{eq: IEI} smaller than \eqref{eq: HLS0}.

\begin{thm}\label{t: SI}
The inequality
\begin{align}
||\nabla \psi||_{L^{2}(\mathbb{R}^{2}_{+})}\leq \sqrt{\frac{2}{c_0\sqrt{\pi}}}||\omega||^{\frac{1}{2}}_{L^{2}(\mathbb{R}^{2}_{+})}||x_2\omega||^{\frac{1}{2}}_{L^{1}(\mathbb{R}^{2}_{+})}   \label{eq: SI}
\end{align}
holds for non-negative $\omega\in L^{2}(\mathbb{R}^{2}_{+})$ such that $x_2\omega\in L^{1}(\mathbb{R}^{2}_{+})$ and $\psi=(-\Delta_D)^{-1}\omega$. The constant $\sqrt{\frac{2}{c_0\sqrt{\pi}}}$ is sharp and its optimizer is the Lamb dipole \eqref{eq: Lamb}. The same inequality holds for $\omega$ without the sign condition.
\end{thm}

The one constraint problem \eqref{eq: MinP} yields a quadratic minimum for $\mu$ and enables one to obtain compactness of the minimizing sequence via the strict subadditivity of the minimum in Lions' concentration-compactness principle, cf. \cite{BNL13}, \cite{B21}, \cite{AC22}. We give a proof for the compactness of the minimizing sequence to \eqref{eq: MinP} using strict subadditivity of the minimum in Appendix A. We show the existence of global weak solutions to \eqref{eq: Euler} for $\zeta_0\in L^{2}(\mathbb{R}^{2}_{+})$ satisfying $x_2\zeta_0\in L^{1}(\mathbb{R}^{2}_{+})$ without assuming the finite mass and give a proof for the stability (Theorem \ref{t: mthm}) in Appendix B.\\

\subsection{Acknowledgements}
KA has been supported by the JSPS through the Grant in Aid for Scientific Research (C) 24K06800, MEXT Promotion of Distinctive Joint Research Center Program JPMXP0723833165, and Osaka Metropolitan University Strategic Research Promotion Project (Development of International Research Hubs). KC has been supported by the NRF grant from the Korean government (MSIT), No. RS-2023-00274499. IJ has been supported by the NRF grant from the Korea government (MSIT), RS-2024-00406821, No. 2022R1C1C1011051.

\section{The variational principle}

We formulate the variational problem \eqref{eq: MinP} and show that all minimizers to \eqref{eq: MinP} are translations of the Lamb dipoles \eqref{eq: Lamb}, and the minimum is the constant \eqref{eq: MinM}. We then deduce Theorem \ref{t: SI} from \eqref{eq: MinM}.

\subsection{The energy inequality}

We set the stream function associated with vorticity in a half plane using the Green function of the Dirichlet problem
\begin{align}
\psi(x)=(-\Delta_{D})^{-1}\omega=\int_{\mathbb{R}^{2}_{+}}G(x,y)\omega(y)\dd y, \quad G(x,y)=\frac{1}{4\pi}\log{\left(1+\frac{4x_2y_2}{|x-y|^{2}} \right)}.  \label{eq: SF} 
\end{align}
By $0<\log(1+t)\lesssim  t^{\alpha}$ for $\alpha\in (0,1]$ and $t> 0$, the Green function satisfies the pointwise bound
\begin{align}
0<G(x,y)\lesssim \frac{x_2^{\alpha}y_2^{\alpha}}{|x-y|^{2\alpha}},\quad x,y\in \mathbb{R}^{2}_{+}.   \label{eq: Green}
\end{align}
We show the energy inequality from the Hardy-Littlewood-Sobolev inequality in $\mathbb{R}^{4}$.

\begin{lem}\label{l: SEI}
The inequality 
\begin{align}
||\nabla \psi ||_{L^{2}(\mathbb{R}^{2}_{+}) }\leq \sqrt[4]{\frac{3}{8\pi}}||x_2 \omega ||_{L^{1}(\mathbb{R}^{2}_{+})}^{\frac{1}{2}} || \omega ||_{L^{2}(\mathbb{R}^{2}_{+})}^{\frac{1}{2}}, \label{eq: SEI}
\end{align}
holds for $\omega\in L^{2}(\mathbb{R}^{2}_{+})$ satisfying  $x_2\omega \in L^{1}(\mathbb{R}^{2}_{+})$ and $\psi=(-\Delta_{D})^{-1}\omega$. 
\end{lem}

\begin{proof}
We apply the Hardy-Littlewood-Sobolev inequality with the sharp constant \cite[Corollary 3.2 (i)]{Lieb83} and H\"older's inequality to estimate  
\begin{align}
\left\|(-\Delta)^{-\frac{1}{2}}g\right\|_{L^{2}(\mathbb{R}^{4}) }
\leq \frac{3^{\frac{1}{4}}}{2^{\frac{5}{4}}\sqrt{\pi}} ||g||_{L^{\frac{4}{3}}(\mathbb{R}^{4}) }
\leq \frac{3^{\frac{1}{4}}}{2^{\frac{5}{4}}\sqrt{\pi}} ||g||_{L^{2}(\mathbb{R}^{4}) }^{\frac{1}{2}}||g||_{L^{1}(\mathbb{R}^{4}) }^{\frac{1}{2}}.  \label{eq: HLS}
\end{align}
For $\psi \in \dot{H}^{1}_{0}(\mathbb{R}^{2}_{+})=\{\psi |\ \nabla \psi\in L^{2}(\mathbb{R}^{2}_{+}),\ \psi(x_1,0)=0\ \}$, we set the function in $\mathbb{R}^{4}$ by 
\begin{align}
\varphi(y)=\frac{\psi(y_4,|y'|)}{|y'|},\quad y=(y',y_4).   \label{eq: ISO}
\end{align}
Then, the map $\dot{H}^{1}_{0}(\mathbb{R}^{2}_{+})\ni \psi\longmapsto \varphi \in \dot{H}^{1}_{\textrm{axi}}(\mathbb{R}^{4})$ is isometrically isomorphic \cite{Yang91} and 
\begin{align*}
||\nabla \varphi||_{L^{2}(\mathbb{R}^{4})}&=\sqrt{4\pi} ||\nabla \psi||_{L^{2}(\mathbb{R}^{2}_{+})},\\
||\Delta \varphi||_{L^{2}(\mathbb{R}^{4})}&=\sqrt{4\pi} \left\|\Delta \psi\right\|_{L^{2}(\mathbb{R}^{2}_{+})},\\
||\Delta \varphi||_{L^{1}(\mathbb{R}^{4})}&=4\pi \left\|x_2\Delta \psi\right\|_{L^{1}(\mathbb{R}^{2}_{+})},
\end{align*}
where $\dot{H}^{1}_{\textrm{axi}}(\mathbb{R}^{4})$ denotes the subspace of axisymmetrinc functions in $\dot{H}^{1}(\mathbb{R}^{4})$. By substituting $g=-\Delta\varphi$ into \eqref{eq: HLS} and using $||(-\Delta)^{1/2}\varphi||_{L^{2}}=||\nabla \varphi||_{L^{2}}$, the inequality \eqref{eq: SEI} follows. 
\end{proof}

\if{
\begin{rem}[Sharp constants]
The sharp constant of the Hardy--Littlewood-Sobolev inequality \eqref{eq: HLS} is attained by the bubble function 

\begin{align*}
g_{b}(y)=\frac{1}{(1+|y|^{2})^{3}}.
\end{align*}\\
The constant of the inequality 

\begin{align}
\left\|(-\Delta)^{-\frac{1}{2}}g\right\|_{L^{2}(\mathbb{R}^{4}) }
\leq \frac{3^{\frac{1}{4}}}{2^{\frac{5}{4}}\sqrt{\pi}} ||g||_{L^{2}(\mathbb{R}^{4}) }^{\frac{1}{2}}||g||_{L^{1}(\mathbb{R}^{4}) }^{\frac{1}{2}},  \label{eq: HLS2}
\end{align}\\
is sharp but not attained by the bubble function. Namely, $B<A$ for the constants

\begin{align*}
A&=\sup_{g\neq 0}\frac{\left\|(-\Delta)^{-\frac{1}{2}}g\right\|_{L^{2}(\mathbb{R}^{4}) }}
{||g||_{L^{\frac{4}{3}}(\mathbb{R}^{4}) } }
=\frac{\left\|(-\Delta)^{-\frac{1}{2}}g_b\right\|_{L^{2}(\mathbb{R}^{4}) }}
{||g_b||_{L^{\frac{4}{3}}(\mathbb{R}^{4}) } }=\frac{3^{\frac{1}{4}}}{2^{\frac{5}{4}}\sqrt{\pi}},\\
B&=\frac{\left\|(-\Delta)^{-\frac{1}{2}}g_b\right\|_{L^{2}(\mathbb{R}^{4}) }}
{||g_b||_{L^{2}(\mathbb{R}^{4}) }^{\frac{1}{2}}||g_b||_{L^{1}(\mathbb{R}^{4}) }^{\frac{1}{2}}}=\frac{5^{\frac{1}{4}}}{2\sqrt{3\pi}}. 
\end{align*}\\
For the Lamb dipole \eqref{eq: Lamb}, $g_{L}=-\Delta \varphi_{L}$ and $\varphi_{L}=\psi_{L}/x_2$ give 

\begin{align*}
 C=\frac{\left\|(-\Delta)^{-\frac{1}{2}}g_L\right\|_{L^{2}(\mathbb{R}^{4}) }}
{||g_L||_{L^{2}(\mathbb{R}^{4}) }^{\frac{1}{2}}||g_L||_{L^{1}(\mathbb{R}^{4}) }^{\frac{1}{2}}}
=  \frac{||\nabla \psi_{L}||_{L^{2}}}{(4\pi)^{\frac{1}{4}}||x_2\omega_L||_{L^{1}}^{\frac{1}{2}}||\omega_L||_{L^{2}}^{\frac{1}{2}} }=\frac{1}{\sqrt{c_0\pi}}.
\end{align*}\\
It turns out that $B<C<A$ by $2^{\frac{5}{2}} /\sqrt{3}=3.2660\cdots<c_0=3.8317\cdots<12/\sqrt{5}=5.3666\cdots$, and the bubble $g_b$ does not maximize the ratio $\left\|(-\Delta)^{-\frac{1}{2}}g \right\|_{L^{2}(\mathbb{R}^{4}) }/ \sqrt{||g ||_{L^{2}(\mathbb{R}^{4})} ||g ||_{L^{1}(\mathbb{R}^{4}) }}$ among non-negative functions. We will show that $C$ is the maximum of the ratio among non-negative functions.
\end{rem}
}\fi

\subsection{The stream function estimates}

We estimate the stream function $\psi=(-\Delta_{D})^{-1}\omega$ for $\omega\in L^{2}(\mathbb{R}^{2}_{+})$ satisfying $x_2\omega \in L^{1}(\mathbb{R}^{2}_{+})$ and express the kinetic energy by using the Green function.

\begin{prop}
Let $1\leq r\leq 2$. The inequality 
\begin{align}
||x_2^{\alpha}\omega||_{L^{r}(\mathbb{R}^{2}_{+})}\leq ||x_2\omega||_{L^{1}(\mathbb{R}^{2}_{+})}^{\alpha}||\omega||_{L^{2}(\mathbb{R}^{2}_{+})}^{1-\alpha},\quad \alpha=\frac{2}{r}-1,  \label{eq: IE}
\end{align}
holds for $\omega\in L^{2}(\mathbb{R}^{2}_{+})$ such that $x_2\omega\in L^{1}(\mathbb{R}^{2}_{+})$.
\end{prop}

\begin{proof}
For $p=1/(2-r)$ and $1/p+1/q=1$, we apply H\"older's inequality to estimate 
\begin{align*}
\int_{\mathbb{R}^{2}_{+}}|x_2^{\alpha}\omega|^{r}dx=\int_{\mathbb{R}^{2}_{+}}x_2^{2-r}|\omega|^{2-r}|\omega|^{2r-2}dx
\leq ||x_2\omega||_{L^{1}}^{\frac{1}{p}}||\omega||_{L^{2}}^{\frac{2}{q}}.
\end{align*}
By taking the $1/r$-th power of both sides, \eqref{eq: IE} follows.
\end{proof}

\begin{prop}
Let $1\leq r\leq 3/2$. The inequality 
\begin{align}
\left\| \frac{\psi}{x_2^{\alpha}} \right\|_{L^{p}(\mathbb{R}^{2}_{+})}\leq C||x_2^{\alpha}\omega||_{L^{r}(\mathbb{R}^{2}_{+})},\quad \frac{1}{p}=\frac{3}{r}-2,\quad \quad \alpha=\frac{2}{r}-1,   \label{eq: SFE}
\end{align}
holds for $\omega\in L^{2}(\mathbb{R}^{2}_{+})$ such that $x_2\omega\in L^{1}(\mathbb{R}^{2}_{+})$ and $\psi=(-\Delta_D)^{-1}\omega$.
\end{prop}

\begin{proof}
By the Green function estimate \eqref{eq: Green} and zero extention of $\omega$ to $x_2<0$, 
\begin{align*}
\left|\frac{\psi(x)}{x_2^{\alpha}}\right|\lesssim \int_{\mathbb{R}^{2}_{+}}\frac{1}{|x-y|^{2\alpha}}y_2^{\alpha}|\omega(y)|dy=\frac{1}{|x|^{2\alpha}}*x_2^{\alpha}|\omega|.
\end{align*}
For $q=1/\alpha$, $|x|^{-2\alpha}\in L^{q,\infty}(\mathbb{R}^{2})$. For $1/p=1/q+1/r-1=3/2-2$, we apply the generalized Young's convolution inequality \cite[p.32]{ReedSimon2} and obtain 
\begin{align*}
\left\|\frac{\psi }{x_2^{\alpha}}\right\|_{L^{p}}\lesssim \left\|\frac{1}{|x|^{2\alpha}}\right\|_{L^{q,\infty}}||x_2^{\alpha}\omega||_{L^{r}}.
\end{align*}
\end{proof}

\begin{lem}\label{p: IBP}
Set $\psi=(-\Delta_{D})^{-1}\omega$ for $\omega\in L^{2}(\mathbb{R}^{2}_{+})$ such that $x_2\omega\in L^{1}(\mathbb{R}^{2}_{+})$. Then, $\sqrt{x_2}\omega\in L^{\frac{4}{3}}(\mathbb{R}^{2}_{+})$ and $\psi/\sqrt{x_2}\in L^{4}(\mathbb{R}^{2}_{+})$ and 
\begin{align}
\int_{\mathbb{R}^{2}_{+}}|\nabla \psi|^{2}dx=\int_{\mathbb{R}^{2}_{+}}\psi\omega dx=\int_{\mathbb{R}^{2}_{+}}\int_{\mathbb{R}^{2}_{+}}G(x,y)\omega(x)\omega(y)dxdy. \label{eq: D}
\end{align}
\end{lem}

\begin{proof}
By \eqref{eq: IE} and \eqref{eq: SFE}, $\sqrt{x_2}\omega\in L^{\frac{4}{3}}(\mathbb{R}^{2}_{+})$ and $\psi/\sqrt{x_2}\in L^{4}(\mathbb{R}^{2}_{+})$ and the right-hand side of \eqref{eq: D} is finite. We take $\theta\in C^{\infty}_{c}[0,\infty)$ such that $\theta=1$ in $[0,1]$ and $\theta=0$ in $[2,\infty)$ and set the cut-off function $\theta_R(x)=\theta(|x|/R)$ for $R\geq 1$. By integration by parts,
\begin{align*}
\int_{\mathbb{R}^{2}_{+}}\psi \theta_R\omega dx
=\int_{\mathbb{R}^{2}_{+}}|\nabla \psi|^{2} \theta_R dx
-\frac{1}{2}\int_{\mathbb{R}^{2}_{+}}\psi^{2}\Delta\theta_Rdx.
\end{align*}
By \eqref{eq: SFE} for $r=6/5$ with $p=2$ and $\alpha=2/3$, $\psi^{2}/x_2^{4/3}\in L^{1}(\mathbb{R}^{2}_{+})$. Since $\Delta\theta_{R}$ is supported in $R\leq |x|\leq 2R$ and satisfies $|\Delta\theta_{R}|\lesssim 1/R^{2}$, the last term converges to zero as $R\to\infty$ and \eqref{eq: D} follows.
\end{proof}

\subsection{Minimization principle}

We define the functional  
\begin{align*}
I_{\lambda}[\omega]=\frac{1}{2\lambda}\int_{\mathbb{R}^{2}_{+}}\omega^{2}dx-\frac{1}{2}\int_{\mathbb{R}^{2}_{+}}\int_{\mathbb{R}^{2}_{+}}G(x,y)\omega(x)\omega(y)dxdy,\quad \lambda>0,
\end{align*}
and the admissible set   
\begin{align*}
K_{\mu}=\left\{\omega\in L^{2}(\mathbb{R}^{2}_{+}) \ \middle|\ \int_{\mathbb{R}^{2}_{+}}x_2\omega dx=\mu,\ \omega\geq 0\  \right\},\quad \mu>0.
\end{align*}
We consider the minimization 
\begin{align}
\mathcal{I}_{\mu,\lambda}=\inf_{K_{\mu}}I_{\lambda}.  \label{eq: MP}
\end{align}

\begin{thm}\label{t: cthm}
Let $\lambda,\mu>0$. The following holds for the minimization problem \eqref{eq: MP}:
\begin{itemize}
\item[(i)] (Compactness) For any minimizing sequence $\{\omega_n\}\subset K_{\mu_n}$ such that $\mu_n\to \mu$ and $I_{\lambda}[\omega_n]\to \mathcal{I}_{\mu,\lambda}$, there exists a sequence $\{y_n\}\subset \partial \mathbb{R}^{2}_{+}$ such that $\{\omega_n(\cdot +y_n)\}$ and $\{x_2\omega_n(\cdot +y_n)\}$ are relatively compact in $L^{2}(\mathbb{R}^{2}_{+})$ and $L^{1}(\mathbb{R}^{2}_{+})$, respectively. In particular, the problem \eqref{eq: MP} has a minimizer. 
\item[(ii)] (Uniqeness) All minimizers of \eqref{eq: MP} are the Lamb dipole \eqref{eq: Lamb} for $\lambda>0$ and $W=\mu \lambda /c_0^{2}\pi$ up to translation for the $x_1$-variable. Moreover, the minimum is given by the constant 
\begin{align}
\mathcal{I}_{\mu,\lambda}=-\frac{1}{2c_0^{2}\pi}\mu^{2}\lambda.   \label{eq: Min}
\end{align}
\end{itemize}
\end{thm}

We show Theorem \ref{t: cthm} (ii) and deduce Theorem \ref{t: SI}. The proof of Theorem \ref{t: cthm} (i) is simpler than that of the compactness argument of \cite{AC22} thanks to the quadratic form of the minimum \eqref{eq: quad} and is given in Appendix A with a minor modification without using the $L^{1}$ estimate.

\subsection{Properties of the minimum}

The Lamb dipole \eqref{eq: Lamb} satisfies the scaling law
\begin{align}
\omega_{L}^{\lambda,W}(x)=W\sqrt{\lambda}\omega_{L}^{1,1}(\sqrt{\lambda}x).  \label{eq: LS}
\end{align}
By scaling $\omega(x)=\lambda \tilde{\omega}(\sqrt{\lambda}x)$,
\begin{align*}
I_{\lambda}[\omega]&=I_{1}[\tilde{\omega}], \\
||x_2 \omega||_{L^{1}(\mathbb{R}^{2}_{+})}&=\frac{1}{\sqrt{\lambda}}||x_2 \tilde{\omega}||_{L^{1}(\mathbb{R}^{2}_{+})}.
\end{align*}
Thus, the minimum $\mathcal{I}_{\mu,\lambda}$ satisfies 
\begin{align}
\mathcal{I}_{\mu,\lambda}=\mathcal{I}_{\mu\sqrt{\lambda},1}.    \label{eq: Scaling}
\end{align}
In the sequel, we consider the case $\lambda=1$ and $\mathcal{I}_{\mu}=\mathcal{I}_{\mu,1}$. We denote the constant in the inequality \eqref{eq: SEI} by $C_{*}=\sqrt[4]{\frac{3}{8\pi}}$.

\begin{lem}
\begin{align}
-\frac{C_{*}^{4}}{8}&\leq \mathcal{I}_{1} <0,  \label{eq: negative} \\
\mathcal{I}_{\mu}&=\mu^{2}\mathcal{I}_{1},\quad \mu\geq 0.  \label{eq: quad} 
\end{align}
In particular, 
\begin{align}
\mathcal{I}_{\mu}<\mathcal{I}_{\mu-\alpha}+\mathcal{I}_{\alpha},\quad 0<\alpha<\mu.  \label{eq: SSA}
\end{align}
\end{lem}

\begin{proof}
We apply the energy inequality \eqref{eq: SEI} for $\omega\in L^{2}(\mathbb{R}^{2}_{+})$ satisfying $x_2\omega\in L^{1}(\mathbb{R}^{2}_{+})$ and $\psi=(-\Delta_D)^{-1}\omega$ and Young's inequality to estimate 
\begin{align*}
\frac{1}{2}||\nabla \psi||_{L^{2}}^{2}\leq \frac{C_{*}^{2}}{2} ||x_2\omega||_{L^{1}}||\omega||_{L^{2}}
\leq \frac{C_{*}^{4}}{8}||x_2\omega||_{L^{1}}^{2}+\frac{1}{2}||\omega||_{L^{2}}^{2}. 
\end{align*}
Thus, the lower bound in \eqref{eq: negative} holds. We take $\omega_1$ such that $||x_2\omega_1||_{L^{1}}=1$. By scaling $\omega_{\sigma}(x)=\sigma^{3}\omega_1(\sigma x)$ for $\sigma>0$, 
\begin{align*}
||x_2\omega_{\sigma}||_{L^{1}}&=||x_2\omega_{1}||_{L^{1}},\\
||\omega_{\sigma}||_{L^{2}}&=\sigma^{2}||\omega_{1}||_{L^{2}}, \\
||\nabla \psi_{\sigma}||_{L^{2}}&=\sigma ||\nabla \psi_{1}||_{L^{2}},\quad \psi_1=(-\Delta_{D})^{-1}\omega_{\sigma}.
\end{align*}
Taking small $\sigma>0$ implies that 
\begin{align*}
\mathcal{I}_{1}\leq \frac{1}{2}||\omega_{\sigma}||_{L^{2}}^{2}-\frac{1}{2}||\nabla \psi_{\sigma}||_{L^{2}}^{2}
=\frac{\sigma^{4}}{2}||\omega_{1}||_{L^{2}}^{2}-\frac{\sigma^{2}}{2}||\nabla \psi_{1}||_{L^{2}}^{2} 
=\frac{\sigma^{2}}{2} \left(\sigma^{2}||\omega_{1}||_{L^{2}}^{2}-||\nabla \psi_{1}||_{L^{2}}^{2} \right)< 0,
\end{align*}
and negativity in \eqref{eq: negative}. Since $K_0=\{0\}$, $\mathcal{I}_{0}=0$. For $\mu>0$, 
\begin{align*}
\mathcal{I}_{\mu}
=\inf\left\{I_1[\omega]\ |\ ||x_2\omega||_{L^{1}}=\mu,\ \omega\geq 0\ \right\} 
=\inf\left\{I_1[\mu \tilde{\omega}]\ |\ ||x_2\tilde{\omega}||_{L^{1}}=1,\ \tilde{\omega}\geq 0\ \right\}
=\mu^{2}\mathcal{I}_{1}.
\end{align*}
Thus, the identity \eqref{eq: quad} holds.
\end{proof}

For the boundedness of minimizing sequences to \eqref{eq: MP}, we prepare an inequality.

\begin{prop}
\begin{align}
||\omega||_{L^{2}}^{2}-4 I_1[\omega]\leq C_{*}^{4}||x_2\omega||_{L^{1}}^{2} \label{eq: VB}
\end{align}
for $\omega\in L^{2}(\mathbb{R}^{2}_{+})$ such that $x_2\omega\in L^{1}(\mathbb{R}^{2}_{+})$.
\end{prop}

\begin{proof}
By \eqref{eq: SEI} and Young's inequality,
\begin{align*}
||\omega||_{L^{2}}^{2}-2I_1[\omega]=||\nabla \psi||_{L^{2}}^{2}\leq C_{*}^{2} ||x_2\omega||_{L^{1}} ||\omega||_{L^{2}}
\leq \frac{1}{2}C_{*}^{4}||x_2\omega||_{L^{1}}^{2}+\frac{1}{2}||\omega||_{L^{2}}^{2}.
\end{align*}
By subtracting $||\omega||_{L^{2}}^{2}/2$ from both side, we obtain \eqref{eq: VB}.
\end{proof}

\subsection{The Euler--Lagrange equation}

We show the uniqueness of minimizers to \eqref{eq: MP} by the uniqueness of solutions to the Euler--Lagrange equation. We set a Banach space $K=\{\omega\in L^{2}(\mathbb{R}^{2}_{+})\ |\ x_2\omega\in L^{1}(\mathbb{R}^{2}_{+})\ \}$ normed with $||\omega||_{K}=\max\{ ||\omega||_{L^{2}},\ ||x_2\omega||_{L^{1}}\}$.

\begin{prop}
The functional $I_1\in C^{1}(K;\mathbb{R})$ satisfies 
\begin{align}
<I_1'[\omega],\eta>=<\omega-\psi,\eta>,\quad \eta\in K,   \label{eq: DF}
\end{align}
for $\omega\in K$ and $\psi=(-\Delta_{D})^{-1}\omega$.
\end{prop}

\begin{proof}
We set $\phi=(-\Delta_{D})^{-1}\eta$ for $\eta\in K$. Then, 
\begin{align*}
I_1[\omega+\varepsilon \eta]
&=\frac{1}{2}\int_{\mathbb{R}^{2}_{+}}|\omega+\varepsilon \eta|^{2}dx 
-\frac{1}{2}\int_{\mathbb{R}^{2}_{+}}|\nabla \psi+\varepsilon \nabla \phi|^{2}dx \\
&=I_1[\omega]+\varepsilon (<\omega,\eta>-<\nabla \psi,\nabla\phi>)+\frac{\varepsilon^{2}}{2}\left(\int_{\mathbb{R}^{2}_{+}}|\eta|^{2}dx-\int_{\mathbb{R}^{2}_{+}}|\nabla\phi|^{2}dx \right).
\end{align*}
By integration by parts, 
\begin{align*}
<D_{G}I[\omega],\eta>=\lim_{\varepsilon\to 0}\frac{I_1[\omega+\varepsilon \eta]-I_1[\omega]}{\varepsilon}
=<\omega-\psi,\eta>.
\end{align*}
Since $\psi/\sqrt{x_2}\in L^{4}(\mathbb{R}^{2}_{+})$ for $\omega\in K$ and $|<\psi,\eta>|\lesssim ||\omega||_{K}||\eta||_{K}$, $D_{G}I[\cdot]\in C(K; K^{*})$ and the Fr\'echet derivative $I'[\omega]=D_{G}I[\omega]$ exists and $I\in C^{1}(K; \mathbb{R})$ satisfies \eqref{eq: DF}. 
\end{proof}

We differentiate the functional $I_1\in C^{1}(K; \mathbb{R})$ at a minimizer $\omega\in K_{\mu}\subset K$.

\begin{prop}\label{p: omega}
Let $\mu>0$. Let $\omega\in K_{\mu}$ be a minimizer of $\mathcal{I}_{\mu}$. Then, there exists $\delta_{0}>0$ such that $|\{x\in \mathbb{R}^{2}_{+}\ |\ \omega> \delta_0\ \}|>0$. Let $h_{*}\in L^{\infty}(\mathbb{R}^{2}_{+})$ be a compactly supported function such that $\textrm{spt}\ h_{*}\subset \{ \omega> \delta_0 \}$ and $\int_{\mathbb{R}^{2}_{+}}xh_{*}dx=1$. Then,
\begin{align}
<I'_1[\omega],\eta>\ \geq 0,\quad \eta=h-\left(\int_{\mathbb{R}^{2}_{+}}x_2h dx \right)h_*,    \label{eq: DF2}
\end{align}
for arbitrary $\delta\in (0,\delta_0)$ and compactly supported functions $h\in L^{\infty}(\mathbb{R}^{2}_{+})$ such that $h\geq 0$ on $\{0\leq \omega\leq \delta\}$.
\end{prop}

\begin{proof}
The minimizer $\omega$ is non-trivial because $0>\mathcal{I}_{\mu}=I_1[\omega]$  
by \eqref{eq: negative}. We take $\delta_0>0$ such that $|\{\omega> \delta_0\}|>0$. We take arbitrary $\delta\in (0,\delta_0)$ and $h$ and set $\eta$ by \eqref{eq: DF2}. Then, $\int_{\mathbb{R}^{2}_{+}}x_2 \eta dx=0$ and $\eta=h$ on $\{\omega\leq \delta_0\}$ by $\textrm{spt}\ h_*\subset \{\omega > \delta_0\}$.

We show that $\omega+\varepsilon \eta\in K_{\mu}$ for sufficiently small $\varepsilon>0$. It suffices to show that $\omega+\varepsilon \eta$ is non-negative in $\mathbb{R}^{2}_{+}$ since $\eta$ is compactly supported. On $\{0\leq \omega\leq \delta\}$, $\eta=h\geq 0$ and $\omega+\varepsilon \eta=\omega+\varepsilon h\geq 0$. On $\{\omega> \delta\}$, $\omega+\varepsilon \eta>\omega-\varepsilon ||\eta||_{L^{\infty}}>0$ for sufficiently small $\varepsilon>0$. Thus $\omega+\varepsilon \eta\in K_{\mu}$.

The inequality \eqref{eq: DF2} follows from $\frac{d}{d\varepsilon}I_1[\omega+\varepsilon \eta]|_{\varepsilon=0} \geq 0$.
\end{proof}

\begin{figure}[h]
 \centering
 \includegraphics[width=0.75\linewidth]{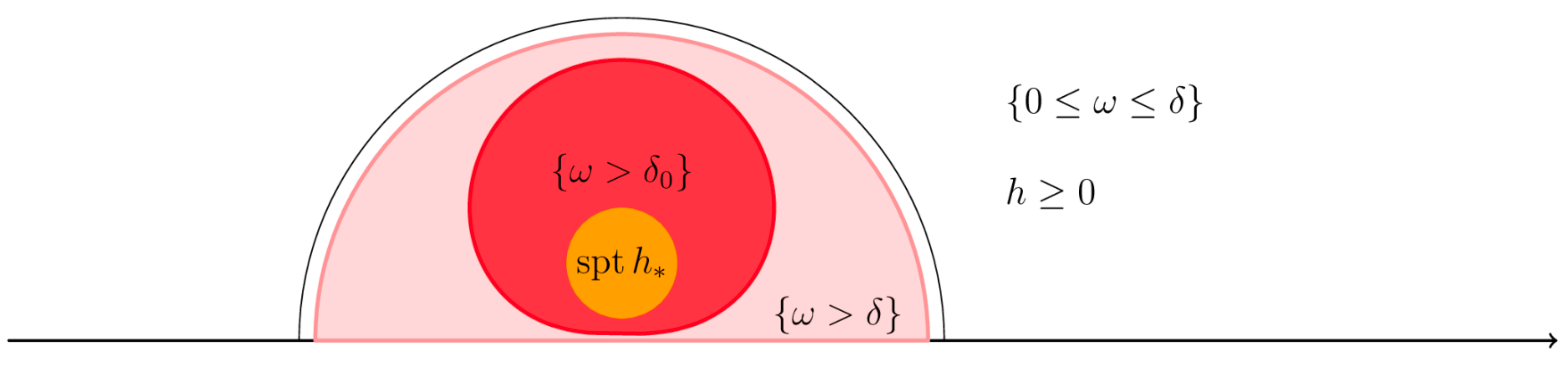}
    \caption{The sets $\{0\leq \omega\leq \delta\}$ and $\{\omega> \delta\}$ in Proposition \ref{p: omega}}
  \label{f: LM}
\end{figure}

\begin{lem}\label{l: EL}
Let $\mu>0$. Let $\omega\in K_{\mu}$ be a minimizer of $\mathcal{I}_{\mu}$. Then, 
\begin{align}
\omega=(\psi-Wx_2)_{+},   \label{eq: EL}
\end{align}
for $W=-2\mu \mathcal{I}_{1}>0$ and $\psi=(-\Delta_{D})^{-1}\omega$.
\end{lem}

\begin{proof}
We set $W=-I_1'[\omega]h_{*}$. By \eqref{eq: DF2}, 
\begin{align*}
0\leq\  <I_1'[\omega],\eta>=\left<I_1'[\omega],h-\left(\int_{\mathbb{R}^{2}_{+}}x_2hdx\right)h_{*}\right>=<I_1'[\omega], h>+W\int_{\mathbb{R}^{2}_{+}}x_2hdx
=<\omega-\psi+Wx_2,h>
\end{align*}
holds for all compactly suppored $h\in L^{\infty}(\mathbb{R}^{2}_{+})$ such that $h\geq 0$ on $\{0\leq \omega\leq \delta\}$. Thus $\omega$ satisfies 
\begin{align*}
\omega-\psi+Wx_2&\geq 0,\quad \textrm{on}\ \{0\leq \omega\leq \delta\},\\
\omega-\psi+Wx_2&=0,\quad \textrm{on}\ \{\omega> \delta\}.
\end{align*}
By letting $\delta\to 0$, 
\begin{align*}
\psi-Wx_2\leq 0&,\quad \textrm{on}\ \{\omega=0\},\\
\omega=\psi-Wx_2&,\quad \textrm{on}\ \{\omega> 0\}.
\end{align*}
Thus, \eqref{eq: EL} holds. We multiply $\omega$ by \eqref{eq: EL} and integrate it. By \eqref{eq: quad}, 
\begin{align*}
0=\frac{1}{2}\int_{\mathbb{R}^{2}_{+}}\left(\omega-(\psi-Wx_2)_+\right)\omega dx
=
\frac{1}{2}\int_{\mathbb{R}^{2}_{+}}\omega^{2}dx
-\frac{1}{2}\int_{\mathbb{R}^{2}_{+}}\psi\omega dx+\frac{W}{2}\int_{\mathbb{R}^{2}_{+}}x_2\omega dx=\mathcal{I}_{\mu}+\frac{W\mu}{2}
=\mu^{2}\mathcal{I}_{1}+\frac{W\mu}{2}.
\end{align*}
Thus, $W=-2\mu \mathcal{I}_{1}>0$.
\end{proof}

\subsection{Uniqueness of minimizers}

We show that minimizers are the Lamb dipoles \eqref{eq: Lamb} by the moving plane method.  

\begin{prop}\label{p: 4D}
Let $\mu>0$. Let $\omega\in K_{\mu}$ be a minimizer of $\mathcal{I}_{\mu}$. For $\psi=(-\Delta_D)^{-1}\omega$, set $\varphi(y)=\psi(y_4,|y'|)/|y'|$ for $y=(y',y_4)\in \mathbb{R}^{4}$. Then, $\varphi\in \dot{H}^{1}_{\textrm{axi}}(\mathbb{R}^{4})\subset L^{4}(\mathbb{R}^{4})$ satisfies $\nabla^{2}\varphi\in L^{2}(\mathbb{R}^{4})$ and 
\begin{align}
-\Delta \varphi=(\varphi-W)_{+},   \label{eq: 4D}
\end{align}
for $W=-2\mu \mathcal{I}_1$. Moreover, $\varphi$ is bounded uniformly continuous up to second order in $\mathbb{R}^{4}$ and 
\begin{align}
\varphi(y)\to 0\quad \textrm{as}\  |y|\to\infty.   \label{eq: Decay}
\end{align}
\end{prop}

\begin{proof}
By the transform $\psi\longmapsto \varphi$, $\dot{H}^{1}_{0}(\mathbb{R}^{2}_{+})$ is isometrically isomophic $\dot{H}^{1}_{\textrm{axi}}(\mathbb{R}^{4})$ and 
\begin{align*}
||\nabla^{2} \varphi||_{L^{2}(\mathbb{R}^{4})}=||\Delta \varphi||_{L^{2}(\mathbb{R}^{4})}=\sqrt{4\pi}||\Delta \psi||_{L^{2}(\mathbb{R}^{2}_{+})}.
\end{align*}
By dividing \eqref{eq: EL} by $x_2$, \eqref{eq: 4D} follows. Since $(\varphi-W)_{+}\leq \varphi$ and $\varphi\in L^{4}(\mathbb{R}^{4})$, applying elliptic $L^{p}$ regularity theory implies that $\varphi$ is locally uniformly bounded in $L^{4}$ up to second orders. Since $W^{2,4}(B)\subset C^{\alpha}(\overline{B})$ for all $0<\alpha<1$, applying elliptic H\"older regularity theory implies that $\varphi$ is bounded uniformly continuous up to second order in $\mathbb{R}^{4}$. By $\varphi\in L^{4}(\mathbb{R}^{4})$, the decay \eqref{eq: Decay} follows. 
\end{proof}

\begin{prop}\label{p: CS}
In Proposition \ref{p: 4D}, set $\Xi=\{y\in \mathbb{R}^{4}\ |\ \varphi(y)>W\ \}$. Then, $\overline{\Xi}$ is compact in $\mathbb{R}^{4}$.
\end{prop}

\begin{proof}
If $\overline{\Xi}$ is unbounded, there exists a sequence $\{y_n\}\subset \Xi$ such that $|y_n|\to\infty$. By \eqref{eq: Decay}, we obtain a contradiction $0<W<\varphi(y_n)\to 0$. 
\end{proof}

\begin{prop}\label{p: RS}
In Proposition \ref{p: 4D}, the function $\varphi(y',y_4+q)$ is radially symmetric and decreasing in $\mathbb{R}^{4}$ for some $q\in \mathbb{R}$.
\end{prop}

\begin{proof}
Since $\varphi$ satisfies the equation \eqref{eq: 4D} with compactly supported $(\varphi-W)_{+}$ and the decay \eqref{eq: Decay}, $\varphi$ is expressed by the Newton potential
\begin{align*}
\varphi(y)=\Gamma*(\varphi-W)_{+}(y)=\int_{\Xi}\Gamma(y-z)(\varphi-W)_{+}(z)dz,
\end{align*}
for $\Gamma(y)=1/(4\pi^{2}|y|^{2})$. By using this representation, we duduce that there exist $p>0$ and $q\in \mathbb{R}$ such that 
\begin{align*}
&\varphi(y',y_4+q)=\frac{p}{|y|^{2}}+g(y), \\
&|g(y)|\leq \frac{C}{|y|^{4}},\quad |\nabla g(y)|\leq \frac{C}{|y|^{5}},\quad |y|\geq 2R+|q|,
\end{align*}
for $R>0$ such that $\Xi\subset B(0,R)$ with some constant $C$ \cite[Lemma 6.2]{AC22}. By this asymptotics, we apply Fraenkel's symmetry result \cite[Theorem 4.2]{Fra00} for positive solutions to the elliptic problem \eqref{eq: 4D}, and deduce that $\varphi(y',y_4+q)$ is radially symmetric and decreasing.  
\end{proof}

\begin{lem}\label{l: U}
Let $\mu>0$. Let $\omega\in K_{\mu}$ be a minimizer of $\mathcal{I}_{\mu}$. Then, $\omega$ is the Lamb dipole \eqref{eq: Lamb} for $\lambda=1$ and $W=\mu/(c_0^{2}\pi)$ up to translation for the $x_1$-variable. Moreover, $\mathcal{I}_{1}=-1/(2c_0^{2}\pi)$. 
\end{lem}

\begin{proof}
For a minimizer $\omega\in K_{\mu}$ of $\mathcal{I}_{\mu}$, we set $\psi=(-\Delta_{D})^{-1}\omega$. By Proposition \ref{p: RS}, $|y|=|x|$ and 
\begin{align*}
\frac{\psi(x_1+q,x_2)}{x_2}=\varphi(y',y_4+q)=\phi(|x|)
\end{align*}
for some decreasing function $\phi$. We may assume $q=0$ by translation. We set 
\begin{align*}
\Psi(x)=\psi(x)-Wx_2=(\phi(r)-W)r\sin\theta=\eta(r)\sin\theta, 
\end{align*}
and $\Omega=\{x\in \mathbb{R}^{2}_{+}\ |\ \psi-Wx_2>0\}=\{x\in \mathbb{R}^{2}_{+}\ |\ \eta(r)>0\ \}$. Since $\overline{\Omega}$ is compact and $\eta$ is decreasing by Propositions \ref{p: CS} and \ref{p: RS}, there exists $a>0$ such that $\Omega=\{x\in \mathbb{R}^{2}_{+}\ |\ |x|<a\ \}$. 

By $-\Delta \Psi=\Psi$ in $\Omega$, $\eta$ satisfies the Bessel's differential equation
\begin{align*}
\eta''+\frac{1}{r}\eta'-\frac{1}{r^{2}}\eta+\eta&=0,\quad \eta>0,\quad 0<r<a,\\
\eta(a)&=0.
\end{align*}
Since $\eta(r)>0$ is bounded at $r=0$ and $\eta(a)=0$, $\eta(r)=C_1J_1(r)$ with some constant $C_1$ and $a=c_0$ for the first zero point $c_0$ of $J_1$. Thus, $\Psi=C_1J_1(r)$ for $r<a$.

By $\Delta \Psi=0$ in $\mathbb{R}^{2}_{+}\backslash \Omega$, $\eta(r)=C_2/r+C_3r$ with some constants $C_2$ and $C_3$. By $\nabla \Psi\to -We_2$ as $|x|\to\infty$, $C_3=-W$. Since $\Psi$ vanishes at $r=a$, $C_2=Wa^{2}$. By continuity of $\partial_r\Psi$ at $r=a$ and $J_1'(c_0)=J_0(c_0)$, $C_1=-2W/J_0(c_0)$. Thus $\omega$ is the Lamb dipole \eqref{eq: Lamb} for $\lambda=1$ and $W=-2\mu \mathcal{I}_{1}$. By the impulse formula \eqref{eq: C}, $\mu=c_0^{2}\pi W$. Thus, $\mathcal{I}_{1}=-1/(2c_0^{2}\pi)$ and $W=\mu/(c_0^{2}\pi)$. 
\end{proof}

\begin{proof}[Proof of Theorem \ref{t: cthm} (ii)]
Let $\mu,\lambda>0$. Let $\omega\in K_{\mu}$ be a minimizer of $\mathcal{I}_{\mu,\lambda}$. By the scaling $\omega(x)=\lambda \tilde{\omega}(\sqrt{\lambda}x)$ and \eqref{eq: Scaling}, $\tilde{\omega}\in K_{\mu\sqrt{\lambda}}$ is a minimizer of $\mathcal{I}_{\mu\sqrt{\lambda}}$. By Lemma \ref{l: U}, $\tilde{\omega}$ is the Lamb dipole $\omega_{L}^{1,\mu\sqrt{\lambda}/(c_0^{2}\pi)}$ \eqref{eq: Lamb} up to translation for the $x_1$-variable. We may assume that $\tilde{\omega}=\omega_{L}^{1,\mu\sqrt{\lambda}/(c_0^{2}\pi)}$. By the scaling law \eqref{eq: LS}, 
\begin{align*}
\tilde{\omega}(x)=
\omega_{L}^{1,\mu\sqrt{\lambda}/(c_0^{2}\pi)}(x)=\frac{\mu\sqrt{\lambda}}{c_0^{2}\pi}\omega^{1,1}_{L}(x).
\end{align*}
We set $W=\mu\lambda/(c_0^{2}\pi)$. By the scaling law \eqref{eq: LS},
\begin{align*}
\omega(x)=\sqrt{\lambda}\tilde{\omega}(\sqrt{\lambda} x)=W\sqrt{\lambda}\omega^{1,1}_{L}(\sqrt{\lambda} x)=\omega^{\lambda,W}_{L}(x).
\end{align*}
The minimum \eqref{eq: Min} follows from $\mathcal{I}_{1,1}=-1/(2c_0^{2}\pi)$ and 
\begin{align*}
\mathcal{I}_{\mu,\lambda}=\mathcal{I}_{\mu\sqrt{\lambda},1}
=-\frac{1}{2c_0^{2}\pi}(\mu\sqrt{\lambda})^{2}=-\frac{1}{2c_0^{2}\pi}\mu^{2}\lambda.
\end{align*}
\end{proof}

\begin{proof}[Proof of Theorem \ref{t: SI}]
For arbitrary non-negative $\omega\in L^{2}(\mathbb{R}^{2}_{+})$ such that $x_2\omega\in L^{1}(\mathbb{R}^{2}_{+})$, we set $\mu=|| x_2\omega||_{L^{1}}$. Then for arbitrary $\lambda>0$, by \eqref{eq: Scaling} and \eqref{eq: quad}, 
\begin{align*}
\lambda\mathcal{I}_{1}||x_2\omega||_{L^{1}}^{2}=\mu^{2}\lambda\mathcal{I}_{1}=\mathcal{I}_{\mu,\lambda}\leq \frac{1}{2\lambda}||\omega||_{L^{2}}^{2}-\frac{1}{2}||\nabla \psi||_{L^{2}}^{2},
\end{align*}
for $\psi=(-\Delta_D)^{-1}\omega$. So we obtain
\begin{align*}
||\nabla \psi||_{L^{2}}^{2}\leq  \frac{1}{\lambda}||\omega||_{L^{2}}^{2}+(-2\mathcal{I}_1)\lambda ||x_2\omega||_{L^{1}}^{2}.
\end{align*}
By taking $\lambda>0$ so that the two terms on the right-hand side are equal, we obtain 
\begin{align*}
||\nabla \psi||_{L^{2}}\leq \sqrt{2\sqrt{-2\mathcal{I}_1}} ||\omega||_{L^{2}}^{\frac{1}{2}} ||x_2\omega||_{L^{1}}^{\frac{1}{2}}.
\end{align*}
By $\mathcal{I}_{1}=-1/(2c_0^{2}\pi)$, \eqref{eq: SI} follows.

For general $\omega\in L^{2}(\mathbb{R}^{2}_{+})$ such that $x_2\omega\in L^{1}(\mathbb{R}^{2}_{+})$, thanks to \eqref{eq: D} and $G(x,y)>0$, 
    \begin{align*}
\int_{\mathbb{R}^{2}_{+}}|\nabla \psi_{\omega}|^{2}dx=\int_{\mathbb{R}^{2}_{+}}\int_{\mathbb{R}^{2}_{+}}G(x,y)\omega(x)\omega(y)dxdy\leq\int_{\mathbb{R}^{2}_{+}}\int_{\mathbb{R}^{2}_{+}}G(x,y) |\omega(x)| |\omega(y)|dxdy =
\int_{\mathbb{R}^{2}_{+}}|\nabla \psi_{|\omega|}|^{2}dx.
\end{align*} 
We apply \eqref{eq: SI} for $\psi_{|\omega|}=(-\Delta_D)^{-1}|\omega|$ and conclude that \eqref{eq: SI} holds without the sign condition for $\omega$.
\end{proof}

\if{

\section{Open questions} 

We mention a few open questions related to the stability of the Lamb dipole \eqref{eq: Lamb}.

The first question is whether the odd-symmetric condition (O) and the non-negative condition (N) are removable from the initial disturbances in the stability theorem of the Lamb dipole (Theorem \ref{t: mthm}). This seems highly challenging. To begin with, it is not clear whether there is a variational characterization of the Lamb dipole in an admissible class without symmetry. Moreover, asymmetric perturbations will, in general, break the non-negative condition (N) immediately, which makes the impulse not coercive. \\

\begin{q}\label{q: 1}
Is the Lamb dipole \eqref{eq: Lamb} stable in the 2D Euler equations \eqref{eq: Euler} for initial disturbances without assuming the conditions (O) and (N)? \\
\end{q}

A relevant question is the stability of Chaplygin's asymmetric dipoles, which is a special relative equilibrium near the Lamb dipole \eqref{eq: Lamb}. This is even harder than Question \ref{q: 1}, as it contains all the difficulties in removing the odd symmetry and non-negative conditions from the Lamb dipole stability.\\

\begin{q}\label{q: 2}
Is the Chaplygin's asymmetric dipole stable in the 2D Euler equations \eqref{eq: Euler} for any disturbances?\\
\end{q}

In the works \cite{CJY}, \cite{AJY}, the stability of Lamb dipoles is investigated in the no-collision regime. If a slower Lamb dipole is located to the right of a faster one, then it is expected that the faster one will catch up with the slower one and collide. See \cite{VF89} for an experimental work. (It is not clear how to define the collision of vortices rigorously). It seems to be a very challenging problem to study what happens as $t\to\infty$, even in asymptotic cases. Since the Euler equations are not completely integrable, we expect the collision to be inelastic, but it could still be almost elastic. Such results are rigorously proved for several dispersive models \cite{MaMeMi10}, \cite{Mu10}, \cite{MaMe11}, \cite{MaMe18}. \\

\begin{q}
Describe the long-time behavior of solutions to the 2D Euler equations initiated from two Lamb dipoles consisting of a slower Lamb dipole located to the right of a faster one. \\
\end{q}

Lastly, it will be interesting to see whether asymptotic stability holds for perturbations of the Lamb dipole. This is a challenging problem since there is generic filamentation, and there might exist quasi-periodic motion close to the Lamb dipole. Actually, we do not even have a good understanding of the set of traveling waves in a neighborhood of the Lamb dipoles.\\

\begin{q}
Is the Lamb dipole asymptotically stable in the 2D Euler equations \eqref{eq: Euler}?\\
\end{q}

}\fi

\appendix

\section{Concentration compactness}

We give a proof for Theorem \ref{t: cthm} (i) by using the concentration-compactness lemma (Lemma \ref{l: CC} or see \cite{Lions84a}, \cite{CL82}, \cite[Lemma 1]{BNL13}, \cite[Lemma 4.1]{AC22}). The compactness argument in \cite{AC22} uses the $L^{1}$ boundedness of the minimizing sequence in all cases (dichotomy, vanishing, compactness) and excludes the possibility of the dichotomy by using Steiner symmetrization. We obtain the compactness of the minimizing sequence to \eqref{eq: MP} by using the strict subadditivity \eqref{eq: SSA} to exclude the possibility of the dichotomy and handle the cases of vanishing and compactness without using the $L^{1}$ boundedness.

\begin{lem}\label{l: CC}
Let $0<\mu<\infty$. Let $\{\rho_n\}\subset L^{1}(\mathbb{R}^{2}_{+})$ satisfy 
\begin{align*}
\rho_n\geq 0\quad n\geq 1,\quad \int_{\mathbb{R}^{2}_{+}}\rho_n\dd x=\mu_n\to \mu\quad \textrm{as}\ n\to\infty.
\end{align*}
There exists a subsequence $\{\rho_{n_k}\}$ satisfying the one of the following:

\noindent 
(i) (Compactness)
There exists a sequence $\{y_k\}\subset \overline{\mathbb{R}^{2}_{+}}$ such that $\rho_{n_k}(\cdot +y_k)$ is tight, i.e., for arbitrary $\varepsilon>0$ there exists $R>0$ such that 

\begin{align}
\liminf_{k\to\infty}\int_{B(y_k,R)\cap \mathbb{R}^{2}_{+}}\rho_{n_k}\dd x\geq \mu-\varepsilon.  \label{Compact}
\end{align}
\noindent 
(ii) (Vanishing) For each $R>0$,
\begin{align}
\lim_{k\to\infty}\sup_{y\in \mathbb{R}^{2}_{+} }\int_{B(y,R)\cap \mathbb{R}^{2}_{+}}\rho_{n_k}\dd x=0.    \label{Vanishing}
\end{align}
\noindent 
(iii) (Dichotomy) There exists $\alpha\in (0,\mu)$ such that for arbitrary $\varepsilon>0$ there exist $k_0\geq 1$ and $\{\rho_{k}^{1}\}$, $\{\rho_{k}^{2}\}\subset L^{1}(\mathbb{R}^{2}_{+})$ such that $\textrm{spt}\ \rho^{1}_{k}\cap \textrm{spt}\ \rho^{2}_{k}=\emptyset$, $0\leq \rho_{k}^{i} \leq \rho_{n_k}$, i=1,2, 
\begin{equation}
\begin{aligned}
&\limsup_{k\to\infty}\left\{||\rho_{n_k}-\rho_{k}^{1}-\rho_{k}^{2}||_{L^{1}}+
\left|\int_{\mathbb{R}^{2}_{+}}\rho_{k}^{1}\dd x-\alpha  \right|
+\left|\int_{\mathbb{R}^{2}_{+}}\rho_{k}^{2}\dd x-(\mu-\alpha)  \right|\right\}
\leq \varepsilon, \\
&\textrm{dist}\ (\textrm{spt}\ \rho^{1}_{k}, \textrm{spt}\ \rho^{2}_{k})\to \infty\quad \textrm{as}\ k\to\infty.
\end{aligned}
\label{Dic}
\end{equation}
\end{lem}

\begin{proof}[Proof of Theorem \ref{t: cthm} (i)]
Let $\{\omega_n\}$ be a minimizing sequence such that $\omega_n \in K_{\mu_n}$, $\mu_n\to \mu$ and $I_{1}[\omega_n]\to \mathcal{I}_{\mu}$ as $n\to\infty$. By \eqref{eq: VB}, $\{\omega_n\}$ is uniformly bounded in $L^{2}$. We set $\rho_n=x_2\omega_n$ and apply Lemma \ref{l: CC}. Then, for a certain subsequence still denoted by $\{\omega_n\}$, one of the following three cases should occur.
\ \\

\noindent
Case 1.\ \textit{Dichotomy:}\\
There exists some $\alpha\in (0,\mu)$ such that for arbitrary $\varepsilon>0$, there exist $k_0\geq 1$ and 
$\{x_2\omega_{1,n}\},\{x_2\omega_{2,n}\}\subset L^{1}$ such that $\omega_{3,n}=\omega_{n}-\omega_{1,n}-\omega_{2,n}$ satisfies $\textrm{spt}\ \omega_{1,n}\cap \textrm{spt}\ \omega_{2,n}=\emptyset$, $0\leq \omega_{i,n}\leq \omega_n$, $i=1,2,3$, and 
\begin{align*}
&\limsup_{n\to\infty}\left\{ ||x_2\omega_{3,n}||_1+|\alpha_n-\alpha|+|\beta_n-(\mu-\alpha)|\right\}\leq \varepsilon, \\
&\alpha_n=\int_{\mathbb{R}^{2}_{+}}x_2\omega_{1,n}\dd x,\quad \beta_n=\int_{\mathbb{R}^{2}_{+}}x_2\omega_{2,n}\dd x,\\
&d_n=\textrm{dist}\ (\textrm{spt}\ \omega_{1,n}, \textrm{spt}\ \omega_{2,n})\to\infty\quad \textrm{as}\ n\to\infty.
\end{align*}
By choosing a subsequence, we may assume that $\alpha_n\to \alpha_{\varepsilon}$ and $\beta_n\to \beta_{\varepsilon}$ and $\sup_{n}||x_2\omega_{3,n}||_1\leq 2\varepsilon$. We set $\omega_n=\tilde{\omega}_{n}+\omega_{3,n}$ and $\psi_n=\tilde{\psi}_{n}+\psi_{3,n}$ by $\tilde{\psi}_{n}=(-\Delta_D)^{-1}\tilde{\omega}_{n}$ and $\psi_{3,n}=(-\Delta_D)^{-1}\omega_{3,n}$. Then, 
\begin{align*}
||\nabla \psi_n||_{L^{2}}^{2}=||\nabla \tilde{\psi}_n||_{L^{2}}^{2}+2(\nabla \tilde{\psi}_n,\nabla \psi_{3,n})_{L^{2}}+||\nabla \psi_{3,n}||_{L^{2}}^{2}.
\end{align*}
By \eqref{eq: SEI},
\begin{align*}
\left|(\nabla \tilde{\psi}_n,\nabla \psi_{3,n})_{L^{2}}\right|
\leq ||\nabla \tilde{\psi}_n||_{L^{2}}||\nabla \psi_{3,n}||_{L^{2}}
&\leq C_{*}^{2}||\tilde{\omega}_n||_{L^{2}}^{\frac{1}{2}}||x_2\tilde{\omega}_n||_{L^{1}}^{\frac{1}{2}}||\omega_{3,n}||_{L^{2}}^{\frac{1}{2}}||x_2\omega_{3,n}||_{L^{1}}^{\frac{1}{2}} \\
&\leq C_{*}^{2}||\omega_n||_{L^{2}}||x_2\omega_n||_{L^{1}}^{\frac{1}{2}}||x_2\omega_{3,n}||_{L^{1}}^{\frac{1}{2}} \\
&\lesssim \sqrt{\mu_n \varepsilon} ||\omega_n||_{L^{2}}.
\end{align*}
Similarly, we estimate $||\nabla \psi_{3,n}||_{L^{2}}^{2}\lesssim \varepsilon ||\omega_n||_{L^{2}}$. We further decompose $\tilde{\psi}_{n}={\psi}_{1,n}+\psi_{2,n}$ by $\psi_{i,n}=(-\Delta_D)^{-1}\omega_{i,n}$ for $i=1$,$2$ and 
\begin{align*}
||\nabla \tilde{\psi}_n||_{L^{2}}^{2}
=||\nabla \psi_{1,n}||_{L^{2}}^{2}
+2(\nabla \psi_{1,n},\nabla \psi_{2,n})_{L^{2}}
+||\nabla \psi_{2,n}||_{L^{2}}^{2}.
\end{align*}
By integration by parts and $G(x,y)\leq \pi^{-1}x_2y_2|x-y|^{-2}$,
\begin{align*}
(\nabla \psi_{1,n},\nabla \psi_{2,n})_{L^{2}}
=(\psi_{1,n},\omega_{2,n})_{L^{2}}
&=\int_{\mathbb{R}^{2}_{+}}\int_{\mathbb{R}^{2}_{+}}G(x,y)\omega_{1,n}(x)\omega_{2,n}(y)dxdy \\
&=\int\int_{|x-y|\geq d_n}G(x,y)\omega_{1,n}(x)\omega_{2,n}(y)dxdy 
\leq \frac{\mu_n^{2}}{\pi d_n^{2}}.
\end{align*}
We thus obtain
\begin{align*}
I_1[\omega_n]
\geq I_1[\omega_{1,n}]+I_1[\omega_{2,n}]-\frac{\mu_n^{2}}{\pi d_n^{2}}-C(\sqrt{\varepsilon}+\varepsilon)
\geq \mathcal{I}_{\alpha_n}+\mathcal{I}_{\beta_n}-\frac{\mu_n^{2}}{\pi d_n^{2}}-C(\sqrt{\varepsilon}+\varepsilon),
\end{align*}
with some constant $C$, independent of $n$. By letting $n\to\infty$,
\begin{align*}
\mathcal{I}_{\mu} \geq \mathcal{I}_{\alpha_{\varepsilon}}+\mathcal{I}_{\beta_{\varepsilon}}-C(\sqrt{\varepsilon}+\varepsilon).
\end{align*}
By letting $\varepsilon\to 0$, $\mathcal{I}_{\mu}\geq \mathcal{I}_{\alpha}+\mathcal{I}_{\mu-\alpha}$. This contradicts the strict subadditivity $\mathcal{I}_{\alpha}+\mathcal{I}_{\mu-\alpha}>\mathcal{I}_{\mu}$.
\ \\

\noindent
Case 2.\ \textit{Vanishing:}\\ 
For each $R>0$,
\begin{align*}
\lim_{n\to\infty}\sup_{y\in \mathbb{R}^{2}_{+} }\int_{B(y,R)\cap \mathbb{R}^{2}_{+}}x_2\omega_n\dd x=0.   
\end{align*}
We shall show $\lim_{n\to\infty}||\nabla \psi_n||_{L^{2}}= 0$ for $\psi_n=(-\Delta_{D})^{-1}\omega_n$. This implies $\mathcal{I}_{\mu}=\liminf_{n\to\infty}I_{1}[\omega_n]=\liminf_{n\to\infty}||\omega_n||_{L^{2}}^{2}/2\geq 0$ and a contradiction to $\mathcal{I}_{\mu}<0$.

We set 
\begin{align*}
||\nabla \psi_n||_{L^{2}}^{2}=\int_{\mathbb{R}^{2}_{+}}dx\int_{|x-y|<R}G(x,y)\omega_n(x)\omega_n(y)dy+\int_{\mathbb{R}^{2}_{+}}dx\int_{|x-y|\geq R}G(x,y)\omega_n(x)\omega_n(y)dy.
\end{align*}
By $G(x,y)\leq \pi^{-1}x_2y_2|x-y|^{-2}$, 
\begin{align*}
\int_{\mathbb{R}^{2}_{+}}dx\int_{|x-y|\geq R}G(x,y)\omega_n(x)\omega_n(y)dy
\leq \frac{\mu_n^{2}}{\pi R^{2}}.
\end{align*}
For $|x-y|<R$ and $G<Rx_2y_2$, 
\begin{align*}
\int_{\mathbb{R}^{2}_{+}}dx\int\limits_{\substack{|x-y|< R, \\ G<  Rx_2y_2}}
G(x,y)\omega_n(x)\omega_n(y)dy\leq R\mu_n \left(\sup_{x\in \mathbb{R}^{2}_{+}}\int_{B(x,R)\cap \mathbb{R}^{2}_{+}}y_2\omega_n(y)dy \right).
\end{align*}
For $|x-y|<R$ and $G\geq Rx_2y_2$, we have $|x-y|\leq  1/\sqrt{R}=:\delta$ and 
\begin{align*}
\int_{\mathbb{R}^{2}_{+}}dx\int\limits_{\substack{|x-y|< R, \\ G\geq   Rx_2y_2}}
G(x,y)\omega_n(x)\omega_n(y)dy
\leq \int_{\mathbb{R}^{2}_{+}}dx\int\limits_{|x-y|< \delta}
G(x,y)\omega_n(x)\omega_n(y)dy.
\end{align*}
By $\log(1+t)\lesssim t^{\alpha}$ for $\alpha\in (0,1]$ and all $t\geq 0$, 
\begin{align*}
G(x,y)\lesssim \frac{x_2^{\alpha}y_2^{\alpha}}{|x-y|^{2\alpha}}.
\end{align*}
For $1<r<2$ and the conjugate exponent $r'$, we apply the Young's convolution inequality for $1/r'=1/q+1/r-1$ with $\alpha q <1$ to estimate
\begin{align*}
\int_{\mathbb{R}^{2}_{+}}dx\int\limits_{|x-y|< \delta}
G(x,y)\omega_n(x)\omega_n(y)dy
&\lesssim 
\int_{\mathbb{R}^{2}_{+}}dx\int_{\mathbb{R}^{2}_{+}}
\frac{1}{|x-y|^{2\alpha}}1_{B(0,\delta)}(x-y)x_2^{\alpha}\omega_n(x)y_2^{\alpha}\omega_n(y)dy \\
&\leq \left\|\frac{1}{|x|^{2\alpha}}1_{B(0,\delta)}*(x_2^{\alpha}\omega_n)\right\|_{L^{r'}}||x_2^{\alpha}\omega_n||_{L^{r}} \\
&\leq \left\|\frac{1}{|x|^{2\alpha}}1_{B(0,\delta)}\right\|_{L^{q}}||x_2^{\alpha}\omega_n||_{L^{r}}^{2} \\
&\lesssim \delta^{\frac{2}{q}(1-\alpha q)} \mu_n^{2\alpha}=\left(\frac{1}{R}\right)^{\frac{1}{q}(1-\alpha q)}\mu_n^{2\alpha}.
\end{align*}
By letting $n\to\infty$ and then $R\to\infty$, we obtain $\lim_{n\to\infty}||\nabla \psi_n||_{L^{2}}=0$.
\ \\

\noindent
Case 3.\ \textit{Compactness:}\\
There exists a sequence $\{y_n\}\subset \mathbb{R}^{2}_{+}$ such that for arbitrary $\varepsilon>0$, there exists $R>0$ such that 
\begin{align*}
\liminf_{n\to\infty}\int_{B(y_n,R)\cap \mathbb{R}^{2}_{+}}x_2 \omega_n(x)dx\geq \mu-\varepsilon.
\end{align*}
By translation, we may assume that $y_n=(0,y_{2,n})$.\\
\noindent
(a) $\limsup_{n\to\infty}y_{2,n}=\infty$. We may assume $\lim_{n\to\infty}y_{2,n}=\infty$ by choosing a subsequence. We set
\begin{align*}
||\nabla \psi_n||_{L^{2}}^{2}=
\int_{B(y_n,R)\cap \mathbb{R}^{2}_{+}}\psi_n(x)\omega_n(x)dx 
+\int_{\mathbb{R}^{2}_{+}\backslash B(y_n,R)}\psi_n(x)\omega_n(x)dx.
\end{align*}
By applying \eqref{eq: SFE} for $p=\infty$,
\begin{align*}
\int_{B(y_n,R)\cap \mathbb{R}^{2}_{+}}\psi_n(x)\omega_n(x)dx\leq \frac{\mu_n}{({y_{2,n}-R})^{2/3}}\left\|\frac{\psi_n}{{x_2}^{1/3}} \right\|_{L^{\infty}}.
\end{align*}
By applying \eqref{eq: SFE} for $r=4/3$ with $\alpha=1/2$ and $r'=4$, 
\begin{align*}
\int_{\mathbb{R}^{2}_{+}\backslash B(y_n,R)}\psi_n(x)\omega_n(x)dx
&\leq \left\|\frac{\psi_n}{x_2^{1/2}} \right\|_{L^{4}(\mathbb{R}^{2}_{+}\backslash B(y_n,R))} \left\| x_2^{1/2}\omega_n \right\|_{L^{\frac{4}{3}}(\mathbb{R}^{2}_{+}\backslash B(y_n,R))} \\
&\leq \left\|\frac{\psi_n}{x_2^{1/2}} \right\|_{L^{4}(\mathbb{R}^{2}_{+})} \left\| x_2\omega_n \right\|_{L^{1}(\mathbb{R}^{2}_{+}\backslash B(y_n,R))}^{\frac{1}{2}}
\left\| \omega_n \right\|_{L^{2}(\mathbb{R}^{2}_{+})}^{\frac{1}{2}}.
\end{align*}\\
By $\limsup_{n\to\infty}||x_2\omega_n||_{L^{1}(\mathbb{R}^{2}_{+}\backslash B(y_n,R))}\leq \varepsilon$, letting $n\to\infty$ and $\varepsilon\to0$ imply $\lim_{n\to\infty}||\nabla \psi_n||_{L^{2}}^{2}=0$. This implies $\mathcal{I}_{\mu}=\lim_{n\to\infty}I_1[\omega_n]\geq 0$ and a contradiction to $\mathcal{I}_{\mu}<0$.

\noindent
(b) $\limsup_{n\to\infty}y_{2,n}<\infty$. We may assume $y_{n}=0$ by choosing a large $R>0$. By choosing a subsequence, $\omega_n\rightharpoonup \omega$ in $L^{2}(\mathbb{R}^{2}_{+})$ and 
\begin{align*}
\int_{B(0,R)\cap \mathbb{R}^{2}_{+}}x_2 \omega dx\geq \mu-\varepsilon.
\end{align*}
Thus, $||x_2\omega||_{L^{1}}=\mu$ and $\omega\in K_{\mu}$. 

We shall show that $\lim_{n\to\infty}||\nabla \psi_n||_{L^{2}}=||\nabla \psi ||_{L^{2}}$ for $\psi=(-\Delta_D)^{-1}\omega$. This implies 
\begin{align*}
\mathcal{I}_{\mu}
=\lim_{n\to\infty}I_1[\omega_n]
=\liminf_{n\to\infty}\left(\frac{1}{2}||\omega_n||_{L^{2}}^{2}-\frac{1}{2}||\nabla \psi_n||_{L^{2}}^{2}\right)
\geq \frac{1}{2}||\omega||_{L^{2}}^{2}-\frac{1}{2}||\nabla \psi||_{L^{2}}^{2}=I_1[\omega]\geq \mathcal{I}_{\mu},
\end{align*}
and $\lim_{n\to\infty}||\omega_n||_{L^{2}}=||\omega ||_{L^{2}}$. Thus, $\omega_n\to \omega$ in $L^{2}(\mathbb{R}^{2}_{+})$. By 
\begin{align*}
\limsup_{n\to\infty}\int_{\mathbb{R}^{2}_{+}\backslash B(0,R)}x_2\omega_n dx\leq \varepsilon,
\end{align*}
we have $\int_{\mathbb{R}^{2}_{+}\backslash B(0,R)}x_2\omega dx\leq \varepsilon$ and 
\begin{align*}
\int_{\mathbb{R}^{2}_{+}}x_2|\omega_n-\omega |dx
&=\int_{B(0,R)\cap \mathbb{R}^{2}_{+}}x_2|\omega_n-\omega |dx
+\int_{\mathbb{R}^{2}_{+}\backslash B(0,R) }x_2|\omega_n-\omega |dx \\
&\lesssim R^{2}||\omega_n-\omega||_{L^{2}(B(0,R)\cap \mathbb{R}^{2}_{+})}
+\int_{\mathbb{R}^{2}_{+}\backslash B(0,R) }x_2(\omega_n+\omega)dx.
\end{align*}
By letting $n\to\infty$ and $\varepsilon\to 0$, $x_2\omega_n\to x_2\omega$ in $L^{1}(\mathbb{R}^{2}_{+})$ follows.

We set 
\begin{align*}
||\nabla \psi_n||_{L^{2}}^{2}=
\int_{B(0,R)\cap \mathbb{R}^{2}_{+}}\psi_n(x)\omega_n(x)dx 
+\int_{\mathbb{R}^{2}_{+}\backslash B(0,R)}\psi_n(x)\omega_n(x)dx.
\end{align*}
We use a short-hand notation $B_{+}=B(0,R)\cap \mathbb{R}^{2}_{+}$ and $B_{+}^{c}=\mathbb{R}^{2}_{+}\backslash B(0,R)$. By using $G(x,y)=G(y,x)$,
\begin{align*}
\int_{B_{+}}\psi_n(x)\omega_n(x)dx 
&=\int_{B_{+}}\omega_n(x)\left(\int_{B_{+}}G(x,y)\omega_n(y)dy+\int_{B_{+}^{c} }G(x,y)\omega_n(y)dy \right)dx \\
&=\int_{B_{+}}\int_{B_{+}}G(x,y)\omega_n(x)\omega_n(y)dxdy 
+\int_{\mathbb{R}^{2}_{+}\backslash B(0,R) }\int_{B(0,R)\cap \mathbb{R}^{2}_{+}}G(x,y)\omega_n(y)\omega_n(x)dxdy\\
&\leq \int_{B_{+}}\int_{B_{+}}G(x,y)\omega_n(x)\omega_n(y)dxdy+\int_{B_{+}^{c} } \psi_n(x)\omega_n(x)dx.
\end{align*}
We thus estimate 
\begin{align*}
\left| ||\nabla \psi_n||_{L^{2}}^{2}-\int_{B_{+}}\int_{B_{+}}G(x,y)\omega_n(x)\omega_n(y)dxdy\right|
\leq 2 \int_{B_{+}^{c} } \psi_n(x)\omega_n(x)dx. 
\end{align*}
By applying \eqref{eq: SFE} for $r=4/3$ with $\alpha=1/2$ and $r'=4$, 
\begin{align*}
\int_{B_{+}^{c}}\psi_n(x)\omega_n(x)dx
\leq \left\|\frac{\psi_n}{x_2^{1/2}} \right\|_{L^{4}(B_{+}^{c})} \left\| x_2^{1/2}\omega_n \right\|_{L^{\frac{4}{3}}(B_{+}^{c})} 
\leq \left\|\frac{\psi_n}{x_2^{1/2}} \right\|_{L^{4}(\mathbb{R}^{2}_{+})} \left\| x_2\omega_n \right\|_{L^{1}(B_{+}^{c})}^{\frac{1}{2}}
\left\| \omega_n \right\|_{L^{2}(\mathbb{R}^{2}_{+})}^{\frac{1}{2}}.
\end{align*}
We obtain
\begin{align*}
\limsup_{n\to\infty}\left|  ||\nabla \psi_n||_{L^{2}}^{2}-\int_{B_{+}}\int_{B_{+}}G(x,y)\omega_n(x)\omega_n(y)dxdy\right|
\leq C\varepsilon^{\frac{1}{2}}.
\end{align*}
Similarly, we obtain
\begin{align*}
\left|  ||\nabla \psi ||_{L^{2}}^{2}-\int_{B_{+}}\int_{B_{+}}G(x,y)\omega(x)\omega(y)dxdy\right|
\leq C\varepsilon^{\frac{1}{2}}.
\end{align*}
Since $G(x,y)\in L^{2}(B_+\times B_+)$ and $\omega_n(x)\omega_n(y)\rightharpoonup \omega (x)\omega (y)$ in $L^{2}(B_+\times B_+)$, we obtain $\lim_{n\to\infty}||\nabla \psi_n||_{L^{2}}=||\nabla \psi ||_{L^{2}}$. The proof is now complete.
\end{proof}

\section{Orbital stability}

We prove Theorem \ref{t: mthm}. The existence of odd-symmetric global weak solutions to \eqref{eq: Euler} is known for odd-symmetric initial data $\zeta_0\in L^{2}\cap L^{1}(\mathbb{R}^{2})$ and $x_2\zeta_0\in L^{1}(\mathbb{R}^{2})$ and $\zeta_0\geq 0$ for $x_2\geq 0$ \cite[Proposition 5.1]{AC22}. For $\zeta_0\in L^{2}(\mathbb{R}^{2})$ such that $x_2\zeta_0\in L^{1}(\mathbb{R}^{2})$, $v_0=k*\zeta_0\in L^{2}(\mathbb{R}^{2})$ by the energy inequality \eqref{eq: SEI}. We show odd-symmetric global weak solutions exist without assuming the $L^{1}$-condition for $\zeta_0$.

\subsection{The existence of global weak solutions}

\begin{prop}\label{p: GWS}
For odd-symmetric initial data $\zeta_0\in L^{2}(\mathbb{R}^{2})$ such that $x_2\zeta_0\in L^{1}(\mathbb{R}^{2})$ and $\zeta_0\geq 0$ for $x_2\geq 0$, there exists an odd-symmetric global weak solution $\zeta\in BC([0, \infty);L^{2}(\mathbb{R}^{2}))$ of \eqref{eq: Euler} such that $x_2\zeta\in BC([0, \infty); L^1(\mathbb{R}^{2}))$, $\zeta\geq 0$ for $x_2\geq  0$, 
\begin{align}
\int_{0}^{\infty}\int_{\mathbb{R}^{2}}\zeta (\varphi_t+v\cdot \nabla \varphi)dxdt
=-\int_{\mathbb{R}^{2}}\zeta_0(x)\varphi(x,0)dx,   \label{eq: WF}
\end{align}
for $v=k*\zeta$ and all $\varphi\in C^{\infty}_{c}(\mathbb{R}^{2}\times [0,\infty))$. This weak solution satisfies the conservation 
\begin{align}
||\zeta||_{L^{2}(\mathbb{R}^{2}_{+})}&=||\zeta_0||_{L^{2}(\mathbb{R}^{2}_{+})},   \label{eq: Z}\\
||x_2\zeta||_{L^{1}(\mathbb{R}^{2}_{+})}&=||x_2\zeta_0||_{L^{1}(\mathbb{R}^{2}_{+})},  \label{eq: P} \\
||v||_{L^{2}(\mathbb{R}^{2}_{+})}&=||v_0||_{L^{2}(\mathbb{R}^{2}_{+})}.  \label{eq: E}
\end{align}
\end{prop}

\begin{proof}
For odd-symmetric $\zeta_0\in L^{2}(\mathbb{R}^{2})$ such that $x_2\zeta_0\in L^{1}(\mathbb{R}^{2})$ and $\zeta_0\geq 0$ for $x_2\geq 0$, we take an odd-symmetric sequence $\{\zeta_{0,n}\}\subset  C_c^\infty(\mathbb{R}^{2})$
such that $\zeta_{0,n}\geq 0$ for $x_2\geq 0$, $\zeta_{0,n}\to \zeta_{0}$ in $L^{2}(\mathbb{R}^{2})$ and $x_2\zeta_{0,n}\to x_2\zeta_{0}$ in $L^{1}(\mathbb{R}^{2})$. Then, there exists an odd-symmetric global weak solution $\zeta_{n}\in BC([0, \infty);L^{2}(\mathbb{R}^{2}))$ of \eqref{eq: Euler} for $\zeta_{0,n}$. By \eqref{eq: Z}, $\zeta_{n}$ is uniformly bounded in $L^{\infty}(0,\infty; L^{2})$. For arbitrary $T>0$, we take a subsequence such that 
\begin{align*}
\zeta_n\stackrel{*}{\rightharpoonup} \zeta\quad \textrm{in}\ L^{\infty}(0,T; L^{2}(\mathbb{R}^{2})).
\end{align*}
By \eqref{eq: Z}, \eqref{eq: E}, and the continuous embedding $H^{1}\subset L^{4}$, $v_n$ is uniformly bounded in $L^{\infty}(0,\infty; H^{1})\subset L^{\infty}(0,\infty; L^{4})$. In particular, $v_n\otimes v_n$ is uniformly bounded in $L^{\infty}(0,\infty; L^{2})$. By \eqref{eq: WF}, $v_n$ satisfies 
\begin{align*}
\partial_t v_n+\nabla \cdot \mathbb{P}(v_n\otimes v_n)=0\quad \textrm{on}\ H^{1}(\mathbb{R}^{2})^*,
\end{align*}
for the projection operator from $L^{2}(\mathbb{R}^{2})$ onto its solenoidal subspace. In particular, $v_n$ is uniformly bounded in $L^{\infty}(0,T; H^{-1}(B(0,R)))$ for $H^{-1}(B(0,R))=H^{1}_{0}(B(0,R))^{*}$ and any $R>0$. By Aubin--Lions theorem, there exists a subsequence such that 
\begin{align*}
v_n\to v\quad \textrm{in}\ L^{2}(0,T; L^{2}(B(0,R))). 
\end{align*}
Since $(\zeta_n,v_n)$ satisfies \eqref{eq: WF}, the limit $(\zeta,v)$ also satisfies \eqref{eq: WF} and $v=k*\zeta$. By the Sobolev regularity $v\in L^{\infty}(0,\infty; H^{1}(\mathbb{R}^{2}))$ and the consistency result \cite[Theorem 10.3 (1)]{DL89}, the limit $\zeta$ is a renormalized solution to the transport equation for $v$. Thus, $\zeta\in BC([0,\infty; L^{2})$ and the equality \eqref{eq: Z} holds by the property of the renormalized solution \cite[Theorem 10.3 (2)]{DL89}. The conservation \eqref{eq: P} and \eqref{eq: E} follow from the weak form \eqref{eq: WF} by applying the same cut-off function argument as in \cite{AC22}.
\end{proof}

\subsection{Application to stability}

Let $0<\lambda,W<\infty$ and $\mu=P=c_0^{2}\pi W/\lambda$. We set the distance from the orbit of the Lamb dipole $\omega_{L}=\omega_{L}^{\lambda, W}$ by 
\begin{align*}
d(\zeta,\omega_{L})=
\inf_{y\in \partial\mathbb{R}^{2}_{+}}\left\{\left\|\zeta-\omega_{L}(\cdot+y) \right\|_{L^{2}(\mathbb{R}^{2}_{+})}+\left\|x_2(\zeta-\omega_{L}(\cdot+y)) \right\|_{L^{1}(\mathbb{R}^{2}_{+})}\right\}.
\end{align*}

\begin{proof}[Proof of Theorem \ref{t: mthm}]
We argue by contradiction. Suppose that the assertion of Theorem \ref{t: mthm} were false. Then, there exists $\varepsilon_0>0$ such that for arbitrary $n\geq 1$, there exists $\zeta_{0,n}\in L^{2}(\mathbb{R}^{2}_{+})$ satisfying $x_2\zeta_{0,n}\in L^{1}(\mathbb{R}^{2}_{+})$, $\zeta_{0,n}\geq 0$,
\begin{align*}
\inf_{y\in \partial\mathbb{R}^{2}_{+}}\left\|\zeta_{0,n}-\omega_{L}(\cdot+y) \right\|_{L^{2}(\mathbb{R}^{2}_{+})}+\left|\int_{\mathbb{R}^{2}_{+}}x_2\zeta_{0,n}dx-\mu\right|\leq \frac{1}{n},
\end{align*} 
and the odd-symmetric global weak solutions $\zeta_n(x,t)$ in Proposition \ref{p: GWS} satisfies 
\begin{align*}
d(\zeta_n(t_n),\omega_{L})\geq \varepsilon_0,
\end{align*}
for some $t_n\geq 0$. We may assume that $t_n>0$ and denote the sequence by $\zeta_n=\zeta_n(t_n)$. We take $y_n\in \partial \mathbb{R}^{2}_{+}$ such that $\zeta_{0,n}-\omega_{L}(\cdot +y_n)\to 0$ in $L^{2}(\mathbb{R}^{2}_{+})$ and $x_2(\zeta_{0,n}-\omega_{L}(\cdot +y_n))\to 0$ in $L^{1}(\mathbb{R}^{2}_{+})$. By applying the energy inequality \eqref{eq: SEI} for $\zeta_{0,n}-\omega_{L}(\cdot +y_n)$ and using $I_{\lambda}[\omega_L(\cdot +y_n)]=\mathcal{I}_{\lambda,\mu}$, we find that $I_{\lambda}[\zeta_{0,n}]\to \mathcal{I}_{\lambda,\mu}$. By conservation \eqref{eq: Z}, \eqref{eq: P}, and  \eqref{eq: E}, $\mu_n=||x_2\zeta_n||_{L^{1}}\to \mu$ and $I_{\lambda}[\zeta_n]\to I_{\lambda}[\omega_{L}]=\mathcal{I}_{\lambda,\mu}$.

By Theorem \ref{t: cthm}, there exists $\{y_n\}\subset \partial \mathbb{R}^{2}_{+}$ such that by choosing a subsequence, $\zeta_n(\cdot +y_n)\to \omega_L=\omega_L^{\lambda,W}$ in $L^{2}(\mathbb{R}^{2}_{+})$ and $x_2\zeta_n(\cdot +y_n)\to x_2\omega_L$ in $L^{1}(\mathbb{R}^{2}_{+})$, respectively. Thus,
\begin{align*}
0=\lim_{n\to\infty}\left\{||\zeta_n(\cdot +y_n)-\omega_L||_{L^{2}(\mathbb{R}^{2}_{+})}
+||x_2(\zeta_n(\cdot +y_n)-\omega_L)||_{L^{1}(\mathbb{R}^{2}_{+})}\right\}
\geq \liminf_{n\to\infty}d(\zeta_n,\omega_L)\geq \varepsilon_0>0.
\end{align*}
We obtain a contradiction.
\end{proof}

\bibliographystyle{plain}
\bibliography{ref}

\begin{thebibliography}{10}

\bibitem{AC22}
K.~Abe and K.~Choi.
\newblock {S}tability of {L}amb dipoles.
\newblock {\em {A}rch. {R}ational {M}ech. {A}nal.}, 244:877--917, (2022).

\bibitem{AJY}
K.~Abe, I.-J. Jeong, and Y.~Yao.
\newblock {S}tability for multiple {L}amb dipoles.
\newblock \href{https://arxiv.org/abs/2507.16474}{arXiv:2507.16474}.

\bibitem{Afan}
Y.~D. Afanasyev.
\newblock Formation of vortex dipoles.
\newblock {\em Physics of Fluids}, 18(3):037103, (2006).

\bibitem{AF86}
C.~J. Amick and L.~E. Fraenkel.
\newblock The uniqueness of {H}ill's spherical vortex.
\newblock {\em Arch. Rational Mech. Anal.}, 92:91--119, (1986).

\bibitem{AF88}
C.~J. Amick and L.~E. Fraenkel.
\newblock The uniqueness of a family of steady vortex rings.
\newblock {\em Arch. Rational Mech. Anal.}, 100:207--241, (1988).

\bibitem{BHM}
M.~Berti, Z.~Hassainia, and N.~Masmoudi.
\newblock Time quasi-periodic vortex patches of {E}uler equation in the plane.
\newblock {\em Invent. Math.}, 233(3):1279--1391, (2023).

\bibitem{BCK}
E.~Bru\`{e}, M.~Colombo, and A.~Kumar.
\newblock Flexibility of two-dimensional {E}uler flows with integrable vorticity.
\newblock \href{https://arxiv.org/abs/2408.07934}{arXiv:2408.07934}.

\bibitem{Burton96}
G.~R. Burton.
\newblock Uniqueness for the circular vortex-pair in a uniform flow.
\newblock {\em Proc. Roy. Soc. London Ser. A}, 452:2343--2350, (1996).

\bibitem{Burton05b}
G.~R. Burton.
\newblock Isoperimetric properties of {L}amb's circular vortex-pair.
\newblock {\em J. Math. Fluid Mech.}, 7:S68--S80, (2005).

\bibitem{B21}
G.~R. Burton.
\newblock Compactness and stability for planar vortex-pairs with prescribed impulse.
\newblock {\em J. Differ. Equ.}, 270:547--572, (2021).

\bibitem{BNL13}
G.~R. Burton, H.~J. Nussenzveig~Lopes, and M.~C. Lopes~Filho.
\newblock Nonlinear stability for steady vortex pairs.
\newblock {\em Comm. Math. Phys.}, 324:445--463, (2013).

\bibitem{CLQZZ}
D.~Cao, S.~Lai, G.~Qin, W.~Zhan, and C.~Zou.
\newblock {U}niqueness and stability of steady vortex rings for 3{D} incompressible {E}uler equation.
\newblock \href{https://arxiv.org/abs/2206.10165}{arXiv:2206.10165}.

\bibitem{CQZZ2}
D.~Cao, G.~Qin, W.~Zhan, and C.~Zou.
\newblock Remarks on orbital stability of steady vortex rings.
\newblock {\em Trans. Amer. Math. Soc.}, 376(5):3377--3395, (2023).

\bibitem{CQZZ}
D.~Cao, G.~Qin, W.~Zhan, and C.~Zou.
\newblock Uniqueness and stability of traveling vortex pairs for the incompressible euler equation.
\newblock {\em Ann. PDE}, 11(1), (2025).

\bibitem{CWZ}
D.~Cao, G.~Wang, and B.~Zuo.
\newblock {S}tability of degree-2 {R}ossby-{H}aurwitz waves.
\newblock \href{https://arxiv.org/abs/2305.03279}{arXiv: 2305.03279v3}.

\bibitem{CL82}
T.~Cazenave and P.-L. Lions.
\newblock Orbital stability of standing waves for some nonlinear {S}chr\"{o}dinger equations.
\newblock {\em Comm. Math. Phys.}, 85:549--561, (1982).

\bibitem{Chap1903}
S.~A. Chaplygin.
\newblock One case of vortex motion in fluid.
\newblock {\em Trudy Otd. Fiz. Nauk Imper. Mosk. Obshch. Lyub. Estest.}, 11(11--14), (1903): {E}nglish translation, Chaplygin, S. A., One case of vortex motion in fluid, Regul. Chaotic Dyn., 12, 219--232, (2007).

\bibitem{Choi24}
K.~Choi.
\newblock {S}tability of {H}ill's spherical vortex.
\newblock {\em {C}omm. {P}ure {A}ppl. {M}ath.}, 77:52--138, (2024).

\bibitem{CJ-Lamb}
K.~Choi and I.-J. Jeong.
\newblock Infinite growth in vorticity gradient of compactly supported planar vorticity near {L}amb dipole.
\newblock {\em Nonlinear Anal. Real World Appl.}, 65:Paper No. 103470, 20, (2022).

\bibitem{Choi25}
K.~Choi, I.-J. Jeong, and Y.-J. Sim.
\newblock On existence of sadovskii vortex patch: A touching pair of symmetric counter-rotating uniform vortices.
\newblock {\em Ann. PDE}, 11(2):18, (2025).

\bibitem{CJY}
K.~Choi, I.-J. Jeong, and Y.~Yao.
\newblock Stability of vortex quadrupoles with odd-odd symmetry.
\newblock \href{https://arxiv.org/abs/2409.19822}{arXiv:2409.19822}.

\bibitem{CH09}
H.~J.~H. Clercx and G.~J.~F. van Heijst.
\newblock {T}wo-{D}imensional {N}avier-{S}tokes {T}urbulence in {B}ounded {D}omains.
\newblock {\em Applied Mechanics Reviews}, 62(2):020802 (25 pages), (2009).

\bibitem{CG}
A.~Constantin and P.~Germain.
\newblock {S}tratospheric {P}lanetary {F}lows from the {P}erspective of the {E}uler {E}quation on a {R}otating {S}phere.
\newblock {\em {A}rch. {R}ational {M}ech. {A}nal.}, 245(1):587--644, (2022).

\bibitem{CGLZ}
A.~Constantin, P.~Germain, Z.~Lin, and H.~Zhu.
\newblock The onset of instability for zonal stratospheric flows.
\newblock \href{https://arxiv.org/abs/2503.14191}{arXiv:2503.14191}.

\bibitem{DdPMP2}
J.~D\'{a}vila, M.~del Pino, M.~Musso, and S.~Parmeshwar.
\newblock Global in time vortex configurations for the $2${D} {E}uler equations.
\newblock \href{https://arxiv.org/abs/2310.07238}{arXiv:2310.07238}.

\bibitem{DdPMP}
J.~D{\'a}vila, M.~{del Pino}, M.~Musso, and S.~Parmeshwar.
\newblock Asymptotic properties of vortex-pair solutions for incompressible euler equations in $\mathbb{R}^2$.
\newblock {\em J. Differ. Equ.}, 408:33--63, (2024).

\bibitem{DdpMW}
J.~D{\'a}vila, M.~del Pino, M.~Musso, and J.~Wei.
\newblock Gluing methods for vortex dynamics in {E}uler flows.
\newblock {\em Arch. Ration. Mech. Anal.}, 235(3):1467--1530, (2020).

\bibitem{Den09}
S.~A. Denisov.
\newblock Infinite superlinear growth of the gradient for the two-dimensional {E}uler equation.
\newblock {\em Discrete Contin. Dyn. Syst.}, 23(3):755--764, (2009).

\bibitem{DL89}
R.~J. DiPerna and P.-L. Lions.
\newblock Ordinary differential equations, transport theory and {S}obolev spaces.
\newblock {\em Invent. Math.}, 98:511--547, (1989).

\bibitem{DG}
M.~Dolce and T.~Gallay.
\newblock The long way of a viscous vortex dipole.
\newblock \href{https://arxiv.org/abs/2407.13562}{arXiv:2407.13562}.

\bibitem{FV94}
J.B. Flor and G.~J.~F. Van~Heijst.
\newblock An experimental study of dipolar vortex structures in a stratified fluid.
\newblock {\em J. Fluid Mech.}, 279:101--133, (1994).
\newblock HAL open access repository, 10.1017/S0022112094003836, hal-02140422, Available at: \href{https://hal.science/hal-02140422/file/Flor94.pdf}{https://hal.science/hal-02140422/file/Flor94.pdf}.

\bibitem{Fra00}
L.~E. Fraenkel.
\newblock {\em An introduction to maximum principles and symmetry in elliptic problems}, volume 128.
\newblock Cambridge University Press, Cambridge, 2000.

\bibitem{Ga11}
T.~Gallay.
\newblock Interaction of vortices in weakly viscous planar flows.
\newblock {\em Arch. Ration. Mech. Anal.}, 200:445--490, (2011).

\bibitem{VF89}
G.~J. F.~Van Heijst and J.~B. Flor.
\newblock Dipole formation and collisions in a stratified fluid.
\newblock {\em Nature}, 340:212--215, (1989).

\bibitem{HT}
D.~Huang and J.~Tong.
\newblock Steady contiguous vortex-patch dipole solutions of the 2d incompressible euler equation.
\newblock {\em {A}rch. {R}ational {M}ech. {A}nal.}, 249(4):46, (2025).

\bibitem{ILN03}
D.~Iftimie, M.~C. Lopes~Filho, and H.~J. Nussenzveig~Lopes.
\newblock Large time behavior for vortex evolution in the half-plane.
\newblock {\em Comm. Math. Phys.}, 237:441--469, (2003).

\bibitem{ISG99}
D.~Iftimie, T.~C. Sideris, and P.~Gamblin.
\newblock On the evolution of compactly supported planar vorticity.
\newblock {\em Comm. Partial Differential Equations}, 24:1709--1730, (1999).

\bibitem{JYZ}
I.-J. Jeong, Y.~Yao, and T.~Zhou.
\newblock Small scale formation for $2d$ {E}uler equation without symmetry.
\newblock \href{https://arxiv.org/pdf/2507.15739}{arXiv:2507.15739}.

\bibitem{KS}
A.~Kiselev and V.~\v{S}ver\'{a}k.
\newblock Small scale creation for solutions of the incompressible two-dimensional euler equation.
\newblock {\em Ann. of Math.}, 180(3):1205--1220, (2014).

\bibitem{KrXu21}
R.~Krasny and L.~Xu.
\newblock Vorticity and circulation decay in the viscous {L}amb dipole.
\newblock {\em Fluid Dynamics Research}, 53(1):015514, feb (2021).

\bibitem{Lamb2nd}
H.~Lamb.
\newblock {\em Hydrodynamics}.
\newblock Cambridge Univ. Press., 2nd ed. edition, 1895.

\bibitem{Lamb3rd}
H.~Lamb.
\newblock {\em Hydrodynamics}.
\newblock Cambridge Univ. Press., 3rd ed. edition, 1906.

\bibitem{Lamb}
H.~Lamb.
\newblock {\em Hydrodynamics}.
\newblock Cambridge Univ. Press., 6th ed. edition, 1932.

\bibitem{LR76}
V.~D. Larichev and G.~M. Reznik.
\newblock Two-dimensional solitary rossby waves.
\newblock {\em Dokl. USSR. Acad. Sci.}, 231:1077--1080, (1976).

\bibitem{Lieb83}
E.~H. Lieb.
\newblock Sharp constants in the hardy-littlewood-sobolev and related inequalities.
\newblock {\em Ann. of Math.}, 118(2):349--374, (1983).

\bibitem{Lions84a}
P.-L. Lions.
\newblock The concentration-compactness principle in the calculus of variations. {T}he locally compact case. {I}.
\newblock {\em Ann. Inst. H. Poincar\'{e} Anal. Non Lin\'{e}aire}, 1:109--145, (1984).

\bibitem{MV94}
V.~V. Meleshko and G.~J.~F. van Heijst.
\newblock On {C}haplygin's investigations of two-dimensional vortex structures in an inviscid fluid.
\newblock {\em J. Fluid Mech.}, 272:157--182, (1994).

\bibitem{NiRa}
A.~H. Nielsen and J.~Juul Rasmussen.
\newblock Formation and temporal evolution of the {L}amb-dipole.
\newblock {\em Physics of Fluids}, 9(4):982--991, 04 (1997).

\bibitem{Or92}
P.~Orlandi and G.~J.~F. {van Heijst}.
\newblock Numerical simulation of tripolar vortices in 2d flow.
\newblock {\em Fluid Dynamics Research}, 9(4):179--206, (1992).

\bibitem{PG15}
I.~P{\'e}rez-Garc{\'\i}a.
\newblock Exact solutions of the vorticity equation on the sphere as a manifold.
\newblock {\em Atm{\'o}sfera}, 28(3):179--190, (2015).

\bibitem{PP}
V.~I. Petviashvili and O.~A. Pohkotelov.
\newblock {\em Solitary Waves in Plasmas and in the Atmosphere}.
\newblock Routledge, 2020.

\bibitem{Protas}
B.~Protas.
\newblock On the linear stability of the {L}amb--{C}haplygin dipole.
\newblock {\em Journal of Fluid Mechanics}, 984:A7, (2024).

\bibitem{ReedSimon2}
M.~Reed and B.~Simon.
\newblock {\em Methods of modern mathematical physics. {I}. {F}unctional analysis}.
\newblock Academic Press, New York-London, 1972.

\bibitem{Stern}
M.~E. Stern.
\newblock Minimal properties of planetary eddies.
\newblock {\em Journal of Marine Research}, 33(1), (1975).

\bibitem{Taylor16}
M.~E. Taylor.
\newblock Euler equation on a rotating sphere.
\newblock {\em J. Funct. Anal.}, 270:3884--3945., (2016).

\bibitem{Wang24}
G.~Wang.
\newblock On concentrated traveling vortex pairs with prescribed impulse.
\newblock {\em Trans. Amer. Math. Soc.}, 377:2635--2661, (2024).

\bibitem{Wang25}
G.~Wang.
\newblock Stability of a class of exact solutions of the incompressible {E}uler equation in a disk.
\newblock {\em J. Funct. Anal.}, 289(5):Paper No. 110998, 24, (2025).

\bibitem{Xu}
X.~Xu.
\newblock Fast growth of the vorticity gradient in symmetric smooth domains for 2{D} incompressible ideal flow.
\newblock {\em J. Math. Anal. Appl.}, 439(2):594--607, (2016).

\bibitem{Yang91}
J.~F. Yang.
\newblock Existence and asymptotic behavior in planar vortex theory.
\newblock {\em Math. Models Methods Appl. Sci.}, 1:461--475, (1991).

\bibitem{zlatos2025}
A.~Zlato{\v{s}}.
\newblock Maximal double-exponential growth for the {E}uler equation on the half-plane.
\newblock \href{https://arxiv.org/abs/2507.04198}{arXiv:2507.04198}.

\end{thebibliography}

\end{document}